\documentclass[12pt]{amsart}

\usepackage[dvipsnames]{xcolor}
\usepackage{graphicx}
\usepackage{dsfont}
\usepackage{amsmath}
\usepackage{color}
\usepackage{caption}
\usepackage{subcaption}
\usepackage[T1]{fontenc}
\usepackage{mathtools}
\usepackage{amsfonts,amssymb,amscd}
\usepackage{enumerate}
\calclayout

\newtheorem{thm}{Theorem}[section]
\newtheorem{prop}[thm]{Proposition}
\newtheorem{lem}[thm]{Lemma}
\newtheorem{cor}[thm]{Corollary}

\newtheorem*{question}{Question}

\theoremstyle{definition}

\newcommand{\uhp}{\mathbb{H}_u} % upper half-plane
\newcommand{\rhp}{\mathbb{H}_r} % right half-plane

\makeatletter
\renewcommand*\env@matrix[1][\arraystretch]{%
  \edef\arraystretch{#1}%
  \hskip -\arraycolsep
  \let\@ifnextchar\new@ifnextchar
  \array{*\c@MaxMatrixCols c}}
\makeatother

\theoremstyle{remark}

\theoremstyle{question}

\numberwithin{equation}{section}

\newcommand{\R}{\mathbb{R}}  % The real numbers.

\newcommand{\Z}{\mathbb{Z}}  % The integers numbers.
\newcommand{\N}{\mathbb{N}}  % The natural numbers.
\DeclareMathOperator{\diam}{diameter}
\DeclareMathOperator{\area}{area}

\DeclareMathOperator{\dist}{dist} % The distance.

\begin{document}

%\section{Responses to Referee's Comments}

%%%%%%%%%%%%%%%%%%%%%%% file template.tex %%%%%%%%%%%%%%%%%%%%%%%%%
%
% This is a general template file for the LaTeX package SVJour3
% for Springer journals.          Springer Heidelberg 2010/09/16
%
% Copy it to a new file with a new name and use it as the basis
% for your article. Delete % signs as needed.
%
% This template includes a few options for different layouts and
% content for various journals. Please consult a previous iss

\title{Prescribing the Postsingular Dynamics of Meromorphic Functions%\thanks{The first author is supported partially by National Science Foundation Grant DMS 16-08577.}
}


\author{Christopher J. Bishop         \and
        Kirill Lazebnik %etc.
}

\maketitle

\begin{abstract}
We show that any dynamics on any discrete planar sequence $S$ can be realized by the postsingular dynamics of some transcendental meromorphic function, provided we allow for small perturbations of $S$. This work is
 motivated by an analogous result of DeMarco, Koch and McMullen 
in \cite{DeMarco2018} for finite $S$ in the rational setting. The proof contains a method for constructing meromorphic functions with good control over both the postsingular set of $f$ and the geometry of $f$, using the Folding Theorem of \cite{Bis15} and a classical fixpoint theorem \cite{MR1513031}. 

\end{abstract}

\section{Introduction}
\label{introduction}

The \emph{singular} set $S(f)$ of a meromorphic function $f:\mathbb{C}\rightarrow\hat{\mathbb{C}}$ is the collection of values $w$ at which one can not define all branches of the inverse $f^{-1}$ in any neighborhood of $w$. If $f$ is rational, then $S(f)$ coincides with the collection of critical values of $f$. If $f$ is transcendental meromorphic, $f^{-1}$ may also fail to be defined in a neighborhood of an \emph{asymptotic value}. The value $w$ is an asymptotic value of $f$ if there is a curve $\gamma(t)\rightarrow\infty$ for which $f(\gamma(t))\rightarrow w$; for instance the exponential map has one asymptotic value at $0$. In the transcendental setting, the set $S(f)$ coincides with the closure of the collection of critical and asymptotic values. 

The \emph{postsingular} set $P(f)$ of a meromorphic function is the closure of the union of forward iterates of the singular set: $\overline{ \cup_{n=0}^{\infty} f^n(S(f))}$. The singular and postsingular sets play an important rule in the study of the dynamics of $f$, both in the rational and transcendental settings (see for instance \cite{MR1230383} for the rational setting, and \cite{2017arXiv170909095S} for the transcendental setting.) The present work addresses the question of allowable geometries and dynamics for the postsingular sets of meromorphic functions. Our main result states that any postsingular dynamics on any discrete sequence can be realized provided one allows for arbitrarily small perturbations of that sequence:

\begin{thm}\label{main}
Let $S\subset\mathbb{C}$ be a discrete sequence (no finite accumulation points) with $4\leq|S|\leq\infty$, let $h: S \rightarrow S$ be any map, and let $\varepsilon>0$. Then there exists a transcendental meromorphic function $f:\mathbb{C}\rightarrow\widehat{\mathbb{C}}$ and a bijection $\psi:S\rightarrow P(f)$ with $|\psi(s)-s|\rightarrow0$ as $s\rightarrow\infty$, $|\psi(s)-s|\leq\varepsilon$ for all $s\in S$, and $f|_{P(f)}=\psi\circ h\circ\psi^{-1}$.
\end{thm}

\noindent Theorem \ref{main} was inspired by Theorem 1.3 of 
\cite{DeMarco2018}:

\begin{thm}\label{DKM} Let $h : S \rightarrow S$ be an arbitrary map defined on a finite set $S \subset \hat{\mathbb{C}}$ with $|S| \geq 3$. Then there exists a sequence of rigid postcritically finite rational maps $f_n$ such that $|P(f_n)| = |S|$, $P(f_n) \rightarrow S$ and $f_n|P(f_n) \rightarrow h|S$ as $n\rightarrow\infty$.\end{thm}

The proof of this result in \cite{DeMarco2018}
 uses iteration on Teichm\"uller space, whereas the
 proof of Theorem \ref{main} uses a fixpoint theorem
 \cite{MR1513031} and quasiconformal folding 
methods developed in \cite{Bis15} which we will 
discuss at length in Section 2.
Quasiconformal folding is a method of
associating entire functions to certain infinite
planar graphs introduced in \cite{Bis15}, 
and was applied there  to construct
various new examples, such as
a wandering domain for an entire function in  the  Eremenko-Lyubich class.
Other applications have been given  by
Fagella, Godillon and  Jarque \cite{MR3339086},
 Fagella, Jarque and Lazebnik \cite{2017arXiv171110629F},
Lazebnik \cite{MR3579902}, \cite{2019arXiv190410086L},
Osborne and Sixsmith \cite{MR3525384},
and  Rempe-Gillen  \cite{arc-like}.
We will review the basic folding construction in  
Section 2.

We now briefly sketch the proof of Theorem \ref{main}, leaving details and some special considerations to subsequent sections. We refer to Figure \ref{fig:example_of_S}. Recall we are given a discrete sequence $S=(s_n)$ and a map $h:S\rightarrow S$. We construct an infinite graph $G$ by enclosing points $s_i\in S$ by disjoint Euclidean discs $D_i$ centered at $s_i$. As discussed in Sections \ref{folding_first_section} and \ref{Graph_I}, we will associate a quasiregular function $g:\mathbb{C}\rightarrow\widehat{\mathbb{C}}$ to the graph $G$. For now, we give the definition of $g$ in a disc $D_i$ under the assumption that $h(s_i)=s_j\in\mathbb{D}$. If $z\in D_i$, then $g(z):=\rho_i \circ (z\mapsto z^d) \circ \tau_i(z)$, where $d\in2\mathbb{N}$, $\tau_i:D_i\rightarrow\mathbb{D}$ is a Euclidean similarity (so $\tau_i(s_i)=0$), and $\rho_i$ is a quasiconformal self-map of $\mathbb{D}$ which is conformal in $(3/4)\mathbb{D}$ and $\rho|_{\partial\mathbb{D}}=\textrm{id}$. The resulting quasiregular map $g$ will have a critical value at $\rho_i(0)$ coming from the critical point $s_i$ in $D_i$, and we will denote this critical value by $s_j^*$. The critical value $s_j^*$ should be thought of as a complex parameter in a small neighborhood of $s_j$ for now, and $s_j^*$ will eventually correspond to $\psi(s_j)$ where $\psi:S\rightarrow P(f)$ is the bijection of Theorem \ref{main}. We note that the definition of $g$ on $\mathbb{C}\setminus D_i$ will not depend on a choice of $s_j^*$.

%{\color{red} After conjugating by an appropriate M\"obius transformation, one may assume that $\pm1\in S$.} 

% and so we naturally define $\rho_i(0)=1/h(s_i)$. We further arrange for the asymptotic values of $g$ to be sent to elements of $S$ which lie in $\mathbb{D}$. 

%\scalebox{.8}{\input{main_theorem_strategy}}

%\begin{figure}
%\centering
%\scalebox{.6}{\input{main_theorem_strategy_larger.pdf_tex}}
%\scalebox{.6}{\input{main_theorem_converted.png}}
%\scalebox{.6}{\input{main_theorem_strategy_larger2.png}}
%\caption{ Illustrated is the general strategy in the proof of Theorem \ref{main}. One applies the Folding Theorem to a graph $G$. Here $s_i\in S$ is the center of $D_i$, and $s_i^*$ is a critical value of $g$. The critical point $s_i$ of $g$ is sent to a critical value $s_j^*$ near $s_j=h(s_i)$. One then arranges (using a fixpoint theorem) for $s_j^*$ to be chosen so that $\phi^{-1}(s_j^*)=s_j$.}
%\label{fig:example_of_S}       % Give a unique label
%\end{figure}

\begin{figure}[htp]
\centerline{
\includegraphics[height=3in]{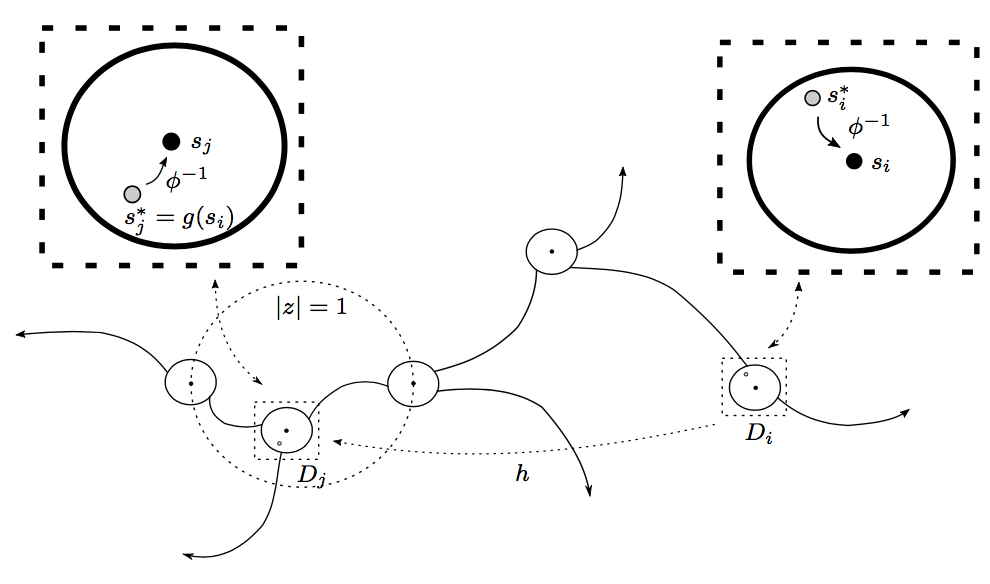}
}
  \caption{  Illustrated is the general strategy in the proof of Theorem \ref{main}. One applies the Folding Theorem to a graph $G$. Here $s_i\in S$ is the center of $D_i$, and $s_i^*$ is a critical value of $g$. The critical point $s_i$ of $g$ is sent to a critical value $s_j^*$ near $s_j=h(s_i)$. One then arranges (using a fixpoint theorem) for $s_j^*$ to be chosen so that $\phi^{-1}(s_j^*)=s_j$.
  }
  \label{fig:example_of_S}     
\end{figure}

Next we apply the measurable Riemann mapping theorem to obtain a quasiconformal map $\phi$ so that $g\circ\phi^{-1}$ is holomorphic. The crux of the proof of Theorem \ref{main} is to arrange for $s_j^*$ to be chosen so that $\phi^{-1}(s_j^*)=s_j\in S$, over all $j$. Indeed, then for $f=g\circ\phi^{-1}$ we would have $f(S(f))\subset S(f)$ with the desired dynamics, since $g\circ\phi^{-1}(s_j^*)=g(s_j)=(h(s_j))^*$ whence again one would have $\phi^{-1}((h(s_j))^*)=h(s_j)$. 

How do we arrange for the parameters $(s_j^*)$ to be chosen so that $\phi^{-1}(s_j^*)=s_j$ over all $j$? Let us consider for now the simpler problem of arranging for $\phi^{-1}(s_j^*)=s_j$ for some fixed, single index $j$. Of course the Beltrami coefficient $\mu_g$ of $g$, and hence the map $\phi$, depends on a choice of $s_j^*$; indeed varying the critical value $s_j^*$ varies the dilatation of $\rho_i$ in $\mathbb{D}$ (and hence the dilatation of $g$ in a small neighborhood of $\partial D_i$). However, as explained in Sections \ref{folding_first_section} and \ref{Graph_I}, one can arrange for the (uniformly bounded) dilatation of $g$ to be supported on a neighborhood of $G$ of arbitrarily small area. Hence one may prove that $\phi$ is uniformly close to the identity regardless of our choice of $s_j^*$ in a small neighborhood of $s_j$, say $|\phi(z)-z|<\varepsilon$ over all $z\in\mathbb{C}$.

Now consider moving the parameter $s_j^*$ continuously in $\overline{D(\varepsilon,s_j)}$. Namely, for each choice of $w\in \overline{D(\varepsilon,s_j)}$, we set $s_j^*:=w$, and we have some resulting quasiregular map $g_w$ and correction map $\phi_w$ where we have arranged for $|\phi_w(z)-z|<\varepsilon$ for $z\in\mathbb{C}$, and $\varepsilon$ is independent of $w$. Thus the map $w\rightarrow\phi_w(s_j)$ is a self-map of $ \overline{D(\varepsilon,s_j)}$, and by continuous dependence on parameters (see Theorem \ref{dependence}), $w\rightarrow\phi_w(s_j)$ is continuous. Thus we can apply a fixpoint theorem (in this instance the classical Brouwer fixpoint theorem) to yield some $w_0\in \overline{D(\varepsilon,s_j)}$ so that by choosing $s_j^*:=w_0$, we have $\phi_{w_0}^{-1}(s_j^*)=s_j$ as needed.

The argument to arrange for the parameters $(s_j^*)$ to be chosen so that $\phi^{-1}(s_j^*)=s_j$ over \emph{all} indices $j$ is similar, however one looks for a fixpoint among a continuous self-mapping of an infinite product of discs centered at the points $s_j$, and one appeals to the following infinite-dimensional fixpoint theorem due to Tychonoff \cite{MR1513031}:

\begin{thm}\label{fixed_point} Let $V$ be a locally convex topological vector space. For any non-empty compact convex set $X$ in $V$, any continuous function $f : X \rightarrow X$ has a fixpoint.
\end{thm}

\noindent For us the locally convex topological vector space of Theorem \ref{fixed_point} will be $\mathbb{C}^{\mathbb{N}}$ (a countable product of complex planes with seminorms $\rho_i((z_j)_{j=1}^\infty):=|z_i|$), and the non-empty compact convex set $X$ will be an infinite product of closed discs containing the points $s_i$ (which is compact by another result of Tychonoff). We remark that a similar fixpoint argument to the one described above was developed independently in \cite{2018arXiv180711820M}. We also record here a statement of continuous dependence on parameters (see, for instance, Theorem 7.5 of \cite{MR1230383}):

%\begin{thm}\label{dependence} Denote by $f^\mu$ the solution of $\frac{\partial f^\mu}{\partial\bar{z}}=\mu\frac{\partial f^\mu}{\partial z}$ fixing $0,1,\infty$ and $QC(\mathbb{C},\mathbb{C})$ the space of quasiconformal self-mappings of $\mathbb{C}$ equipped with the topology of uniform convergence on compact subsets. Then the map $L^\infty(\mathbb{C})\rightarrow QC(\mathbb{C},\mathbb{C})$ given by $\mu\rightarrow f^\mu$ is analytic, in the sense that it is continuous, and for each $z\in\mathbb{C}$ the map $\mu\rightarrow f^\mu(z)$ is analytic. Moreover, for each $z\in\mathbb{C}$ the map $\mu\rightarrow (f^\mu)^{-1}(z)$ is continuous.\end{thm}

\begin{thm}\label{dependence} \emph{(Continuous dependence on parameters)} Let $\mu \in L^{\infty}(\mathbb{C})$ with $||\mu||_{L^\infty(\mathbb{C})}<1$. Denote by $\phi_\mu$ the unique quasiconformal solution of $\frac{\partial \phi_\mu}{\partial\bar{z}}=\mu\frac{\partial \phi_\mu}{\partial z}$ satisfying some fixed normalization. If $\mu_n \rightarrow \mu$ a.e., then $\phi_{\mu_n} \rightarrow \phi_{\mu}$ uniformly on compact subsets. Consequently, for any fixed $z\in\mathbb{C}$, the map $L^\infty(\mathbb{C})\rightarrow\mathbb{C}$ given by $\mu\rightarrow\phi_\mu(z)$ is continuous.\end{thm}

\noindent We remark that in the present work we will only need to consider a subclass of Beltrami coefficients $\mu$ which satisfy a strong \emph{thinness} condition near $\infty$, so that $\phi_{\mu}$ is asymptotically conformal at $\infty$ (see Section 3). For such maps one may normalize $\phi_{\mu}$ such that $\phi_\mu(z)=z+O(1/|z|)$ as $z\rightarrow\infty$, and this is the normalization we will always use in the present work. 

We leave open the following question arising naturally from the statement of Theorem \ref{main}, which asks whether it is necessary, in general, to consider perturbations of the sequence $S$:

\begin{question} Given any discrete planar sequence $S$ and some map $h:S\rightarrow S$, does there always exist a meromorphic $f$ so that $P(f)=S$, and $f|_S=h$?
  
\end{question}

\noindent A similar question was asked for finite $S$ and rational $f$ in \cite{DeMarco2018} (see Question 1.2 of \cite{DeMarco2018}):

\begin{question} Let $S\subset P^1(\overline{\mathbb{Q}})$ be a finite set. Is every map $h: S\rightarrow S$ realized by a rigid rational map $f: P(f) \rightarrow P(f)$ with $P(f)=S$?

\end{question} 

\noindent We also remark that an analogous version of Theorem \ref{main} holds for any infinite sequence in $\widehat{\mathbb{C}}$ with a unique accumulation point (not necessarily at $\infty$); in this case the $f$ produced in Theorem \ref{main} would have one essential singularity at this accumulation point (not necessarily $\infty$). Thus Theorem \ref{main} could be viewed as a statement that Theorem \ref{DKM} of \cite{DeMarco2018} remains true for infinite sets $S$ with a unique accumulation point, provided one is allowed to place an essential singularity of the function $f$ at that accumulation point. It seems plausible, moreover, that any dynamics on a sequence $S$ with $n$ accumulation points could be realized by the postsingular dynamics of a meromorphic function with $n$ essential singularities (one at each accumulation point), provided one allows for perturbations of $S$ as in Theorem \ref{main}. There are further generalizations to be made in this direction.

{\subsection*{Acknowledgements}
\addtocontents{toc}{\protect\setcounter{tocdepth}{1}}

The authors would like to thank the anonymous referees for their suggestions which led to an improved version of the paper. }

%{\color{red} \subsection*{Acknowledgements}
%\addtocontents{toc}{\protect\setcounter{tocdepth}{1}}

%The authors would like to thank the anonymous referees for their suggestions which have led to an improved version of the paper. 
% }

%\section{Quasiconformal folding and meromorphic functions}
%\label{folding_first_section}

\section{Bounded geometry graphs}
\label{bounded geometry section}

In this  and the next section we review the quasiconformal 
folding  method of \cite{Bis15} for
constructing entire functions and adapt it to producing 
meromorphic functions.

Suppose $G$ is an unbounded, locally finite, connected  planar
graph. We say $G$ has  bounded geometry  if
\begin{enumerate}[{(1)}]
\item 
 The edges of $G$ are $C^2$ with uniform bounds.

 \item  The  union of edges meeting at a vertex $v$ 
are a $K$-bi-Lipschitz image of a ``star'' $\{z \in {\mathbb C}: 
0\leq  z^k  \leq r\}$  for some uniformly bounded $k,K$ ($r$ can 
be any positive, finite value; the star consists of $k$ equal 
length segments meeting at evenly spaced angles).

 \item  For any pair of non-adjacent edges $e$ and $f$,
 $\text{diam}(e)/\text{dist}(e,f)$ is uniformly bounded from above.
\end{enumerate}

The values for which these conditions hold are called the 
``bounded geometry constants'' of $G$.
We define a neighborhood $T_\gamma(r)$ of an arc $\gamma$
by 
$$ T_\gamma(r) = \{z\in\mathbb{C}: {\rm{dist}}(z,e) < r \cdot {\rm{diam}}(\gamma) \}, $$
and we define a neighborhood $T(r)$ of $G$ by taking
the union of these neighborhoods where $\gamma$ ranges 
over the edges of $G$.
This is a sort of Hausdorff neighborhood of $G$, but 
adapted to the local geometry of $G$ (the ``thickness'' 
of the neighborhood is proportional to the diameters of 
nearby edges).

 It is sometimes 
helpful to replace condition (1) by a stronger condition
that was introduced in \cite{MR3653246}: 
we say an arc $\gamma$ is $\varepsilon$-analytic if there is a conformal 
map on $T_\gamma(\varepsilon)$ that maps $\gamma$ to a line segment.
 %A vertex  $v$ is called $\varepsilon$-analytic if it has degree 
%2 and the union of the two edges meeting at $v$ form  an 
%$\varepsilon$-analytic arc. 
We say $G$ is uniformly analytic if it has bounded 
geometry and  every edge is $\varepsilon$-analytic 
for some fixed $\varepsilon >0$. 
Note that if we add extra vertices to the edges of a 
uniformly analytic graph  $G$ so as to  form a new bounded 
geometry tree, the new tree is also uniformly analytic with 
the same constant.
All the graphs constructed in this paper will be uniformly 
analytic.

Since $G$ is connected, the connected components 
of $\Omega := {\mathbb C} \setminus G =\cup_j \Omega_j$
 are simply connected.
We further assume that any bounded components are  disks
and the vertices of $G$ are evenly spaced on the boundary 
of each disk. We call these the D-components (for 
``disk components''). To apply the Folding Theorem, 
D-components  need only be 
bounded Jordan domains, but the special case of 
disks is all that we need here, and this extra 
assumption  will simplify the discussion.
Note that this assumption and bounded geometry imply 
that no two D-components touch each other. 
Moreover, we shall assume that every D-component 
contains an even number of vertices on its boundary; this 
is necessary because we will eventually map vertices
of $G$ to $\pm 1$ with edges mapping to top and 
bottom halves of the unit circle and  we will 
need to have an equal number of each type of edge 
and vertex on the boundary of the component.
For each D-component $\Omega_j$ let $\tau_j$ be a
complex-linear map to the unit disk, mapping the vertices of 
$G$ to the $2n^{\textrm{th}}$ roots of unity if there
are  $2n$ vertices on $\partial \Omega_j$. 

The unbounded 
components of $\Omega$ are called R-components (for ``right
half-plane components''). 
For each R-component $\Omega_j$, choose a conformal map $\tau_j$
from $\Omega_j$ to the right half-plane, ${\mathbb H}_r 
=\{ x+iy: x>0\}$, 
taking $\infty$ to $\infty$. We will denote by $\tau: \Omega\rightarrow\mathbb{C}$ the map defined as $\tau_j$ in each component $\Omega_j$ of $\Omega=\mathbb{C}\setminus G$.  We think of each edge 
in $G$ as having two sides, which may belong to the same 
or different complementary components of $G$. The map $\tau$ sends all the sides  belonging to 
a given unbounded component  $\Omega_j$ to intervals that 
partition the imaginary axis. The bounded geometry 
condition implies adjacent intervals
have uniformly 
comparable lengths; this is Lemma 4.1 of \cite{Bis15}.
We call such 
a partition of a line a ``quasisymmetric partition''.
The proof of this lemma given in \cite{Bis15} 
is just a sketch, so 
we  give a more detailed  version here.

\begin{lem}\label{Lemma_4.1} [Lemma 4.1, \cite{Bis15}]
Suppose notation is as above.
If $G$ is a bounded geometry graph, then $\tau$ maps 
sides of each unbounded complementary component $\Omega_j$ to a 
quasisymmetric partition of $\partial \rhp$.
\end{lem} 

\begin{proof}
We will use some simple facts involving 
conformal modulus, e.g., as discussed in  
Chapter IV of \cite{MR2150803}. We first note that if 
$I=[a,b]$ and  $J=[b,c]$ are adjacent intervals 
on the real line, then $I$ and $J$  have comparable lengths
if and only if the conformal moduli of the two  path 
families connecting opposite sides of the quadrilateral 
are bounded (connecting $I$ to $K=[c,\infty)$ and 
connecting $J$ to $L=(-\infty,a]$).

Suppose  $I$ and $J$ correspond to 
sides  $e,f$ of $\Omega_j$ (they might belong to two 
distinct  edges of $G$  or be the two sides of a single edge),
and let $g$ and $h$ be the parts of $\partial 
\Omega_j$ that correspond to the rays $K$ and $L$ respectively.
By rescaling we may assume $e$ has diameter $1$.
We claim 
that there is an $\varepsilon >0$ so that 
any path $\gamma$  connecting $e$ to $g$ inside $\Omega_j$
has length at least $\varepsilon $.  If $\gamma$ connects $e$ to a 
 non-adjacent edge $e'$  this follows immediately 
 from condition (3) in the definition of bounded 
 geometry. Otherwise, $e$, $f$ must be the two sides of a single edge, and $\gamma$ connects $e$ to a point of an 
adjacent edge $e'$  (possibly a point on the other side of $e$). 
If $\gamma$ leaves $N(e, \varepsilon)$, 
the $\varepsilon$-neighborhood of $e$, then 
it obviously has length at least $\varepsilon$.
Otherwise, $\gamma$ remains inside $N(e, \varepsilon)$.
Suppose $v,w$ are the endpoints of $e$.
By  the bounded geometry conditions 
there is a $C < \infty$ and  $\varepsilon>0$ 
so that $N(e, \varepsilon) \setminus (B(v,C\varepsilon)
\cup B(w,C\varepsilon))$ is disjoint from all edges of $G$ 
except $e$ itself.
With $\gamma$ as above, it must pass through one of these 
balls (where the graph has degree 1) and then hit the 
other ball before reaching $g$. Thus $\gamma$
must connect these two balls,
and hence it must  have diameter $\geq 1-O(\varepsilon) > \varepsilon$.

Now define a metric  by $\rho =1/\varepsilon $
on  $N(e, \varepsilon)$, 
the  $\varepsilon$-neighborhood 
of $e$, and zero elsewhere. Since any 
path connecting $e$ to $g$ inside $\Omega_j$
has length at least $\varepsilon$,  $\rho$ is admissible 
for this family.
On other hand, part (1) of  bounded geometry implies that 
$N(e,\delta)$ has area $O(\varepsilon )$, so  
$\iint \rho^2 dxdy = O( \varepsilon^{-1})$, a uniform
bound for the modulus of the path family that 
depends only on the bounded geometry constants.
The same argument applies to $f$ and $h$, proving the lemma.

\end{proof} 

\section{Quasiconformal folding and meromorphic functions}
\label{folding_first_section}

In order to state the Folding Theorem we need another 
assumption on the graph $G$:
 we will  assume that the $\tau$-image
of  every side of every R-component
has  length bounded 
below by $ \pi$; this is the so called 
``$\tau$-condition'' or ``$\tau$ lower bound''.
If there is a conformal map so that the images have 
lengths uniformly bounded away from zero, then by 
multiplying by a positive constant, we may assume the
lower bound is $ \pi$. Thus we usually only need to 
check that some lower bound holds.
For example, it is easy to check that a half-strip 
satisfies the $\tau$-condition  (for some choice 
of $\tau$) if the vertices are
evenly spaced, and even if the gaps between 
vertices decrease exponentially along the edges 
of the strip. Moreover,  this is 
essentially the only case that we will need
to consider in this paper.

This collection of conformal maps on R-components 
and $\mathbb{C}$-linear maps  on D-components defines 
a holomorphic map from $\Omega$ to the right half-plane. 
This map  need not be continuous across $G$, but the 
following result says that it can be modified in a 
neighborhood of $G$ so that it becomes continuous 
on the whole plane and is not far from holomorphic (it 
is quasi-regular). The following is a special case
of the result proven by the first author in \cite{Bis15}:

\begin{thm}[Folding Theorem] \label{QC folding thm}
Suppose notation and assumptions are as above.
Then there  are constants $r,K< \infty$ (that depend
only on the bounded geometry constants of $G$) and 
a graph $G'$  so that 
\begin{enumerate}
\item  $G'$ is obtained from $G$ by adding  a finite 
       number of finite trees 
       to the vertices of $G$ (the number added at any 
       vertex is at most the degree of that vertex).
\item 
       Each added tree is contained inside $T(r)$.
\item Each added tree  is contained in an R-component, except 
      for the vertex it shares with $G$. Therefore 
      each complementary component $ \Omega_j'$ of 
      $G'$  is contained in a complementary component $\Omega_j$
      of $G$ and this is a one-to-one correspondence.
      Note that $\Omega_j \setminus \Omega_j' \subset 
      \Omega_j \cap T(r)$.
\item For each R-component $ \Omega_j'$ of $G'$, there 
    is a $K$-quasiconformal map $\eta_j$ of $\Omega_j'$ 
     to the right half-plane
    that maps the sides of $\Omega_j'$ to intervals 
    of length $ \pi$ on the imaginary axis.
\item   On each side of an R-component  $  \Omega_j'$, 
    the map $\eta_j$  multiplies 
    arclength by a constant factor (which must 
    be $\pi$ divided by the length of that side).
\item For each  R-component $\Omega_j$,  $\eta_j = \tau_j$ 
      on $\Omega_j \setminus T(r)$. In particular, $\eta_j$ 
      is conformal off $T(r)$.
\end{enumerate}
\end{thm}

We define a map $F$ on $\Omega' = {\mathbb C} \setminus G'
= \cup_j \Omega_j'$ 
by setting $F= \exp \circ \eta_j$ on each R-component $\Omega_j'$
and setting $F=\tau^n$ on a  D-component that has $2n$
vertices.   Because the only closed loops in $G$ are 
the boundaries of the D-components, and because we 
have assumed each of these contains an even number of 
vertices, it is easy to check that $G$ is bipartite 
and we choose a labeling of its vertices by $\pm 1$, 
so that adjacent vertices always have different labels.
By post-composing with a translation (for  R-components) 
 or a rotation  (for D-components) we can assume $F$ 
maps each vertex of $G$ to $\pm 1$, agreeing with its
label.

 Note that each new unbounded component $\Omega_j'$ 
lies inside one of the old R-components, and we will 
call these the new R-components 
(to distinguish them from the original R-components).
Note that $F$ extends continuously across any edge bounding both 
a D-component and a new R-component. This follows  since 
both maps send the edge to the same half of 
 the unit circle, with both  maps agreeing at the 
endpoints (which map to  $\pm 1 $), and both
maps  multiply 
arclength  by the same constant factor.

The same observation shows that for any point on 
an edge bounding 
two  new R-components (or an edge for which both sides belong 
to the same new R-component), the two possible images
under $F = \exp \circ \eta$ are conjugate points on 
the unit circle. To ``close the gap'', we  define 
a $3$-quasiconformal map $\sigma$  from $\{|z|>1\}$ to 
${\mathbb C}\setminus [-1,1]$ as follows. 
Use a M{\"o}bius transformation $\mu (z) = 
(z+1)/(z-1)$ to map $\{|z|>1\}$
to the right half-plane with $\{-1,+1\}$ mapping 
to $\{0,\infty\}$ (also note that $\mu$ is its 
own inverse). Consider the $3$-quasiconformal  map $\nu$ 
from the right half-plane to ${\mathbb C} \setminus (-\infty,0]$
 that is the 
identity on $\{|\arg(z)|\leq \pi/4\}$ and which triples
angles in the two remaining sectors. Post-composing with 
$\mu^{-1}$ gives the desired map $\sigma$.

\begin{figure}[htp]
\centerline{
\includegraphics[height=1.5in]{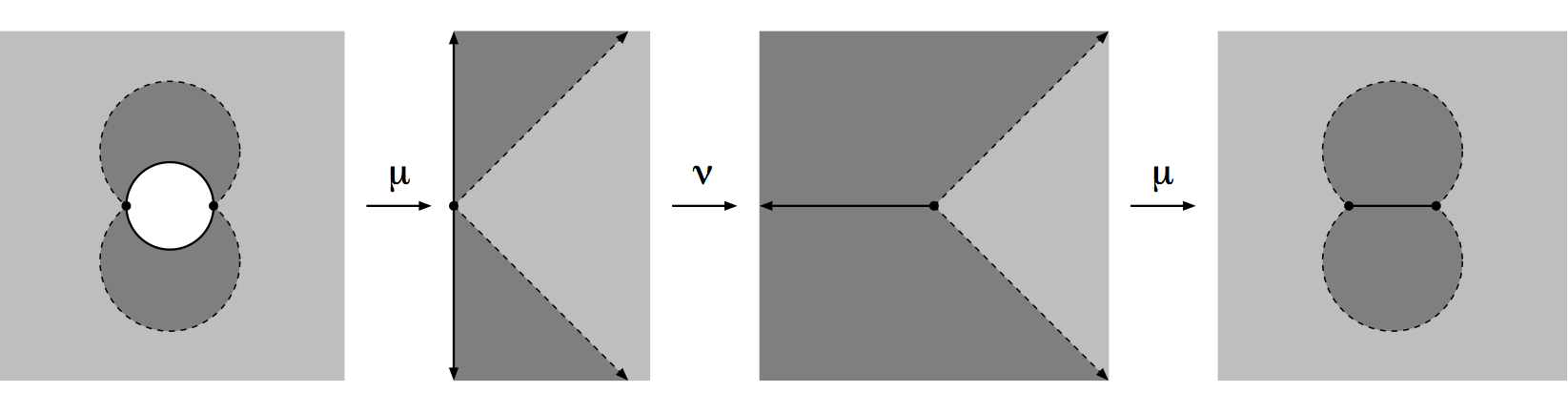}
}
  \caption{  The 3-quasiconformal map that sends the outside 
  of the unit  disk to the outside of $[-1,1]$ and identifies
  conjugate points on the circle. The map is conformal (indeed, 
  is the identity) on the light gray region.
  }
  \label{geodesic fig}
\end{figure}

We gave the definition  in sectors  so that $\sigma$
is conformal off a bounded neighborhood of $\{|z|=1\}$
(it would have  been easier to define a 2-quasiconformal 
map with the same boundary values, but non-conformal 
in the whole plane). Thus applying $\sigma$ will keep our 
map holomorphic outside $T(r_1)$ if $r_1$ is large enough.

Note that $\sigma$ maps two conjugate points on the 
unit circle to the same point of $[-1,1]$, so 
$\sigma \circ F$ will extend continuously across 
the edge we are considering. Here $\sigma \circ F$ is applied 
on the components of the $F$-pre-image of 
${\mathbb C}\setminus (\{|z|\leq 1\} \cup[1,\infty))$ that contain 
the relevant edges of $G$ on their boundary, where we note that $\sigma(z)=z$ for $z\in[1,\infty)$.

%{\color{red} The map  $\sigma \circ F$ is applied 
%on the components of the $F$-pre-image of 
%${\mathbb C}\setminus (\{|z|\leq 1\} \cup[1,\infty))$. The map $\sigma \circ F$ is continuous across edges since $\sigma$ maps two conjugate points on the 
%unit circle to the same point of $[-1,1]$, and $\sigma \circ F$ is continuous across other  }

%Note that $\sigma$ maps two conjugate points on the 
%unit circle to the same point of $[-1,1]$, so 
%$\sigma \circ F$ will extend continuously across 
%an edge which does not neighbor a D-component. Here $\sigma \circ F$ is applied 
%on the components of the $F$-pre-image of 
%${\mathbb C}\setminus (\{|z|\leq 1\} \cup[1,\infty))$ that contain 
%such edges of $G'$ on their boundary, and we apply $F$ to $F$-preimages of ${\mathbb C}\setminus (\{|z|\leq 1\} \cup[1,\infty))$ which contain an edge neighboring a D-component. We remark that the definitions of $F$ and $\sigma \circ F$ match up on common boundaries of $F$-preimages  of ${\mathbb C}\setminus (\{|z|\leq 1\} \cup[1,\infty))$ sincer $\sigma(z)=z$ for $z\in\mathbb{R}\setminus(-1,1)$. 

%{\color{red} where we note that $\sigma(z)=z$ for $z\in\mathbb{R}\setminus(-1,1)$}.

Finally,  given $ w \in {\mathbb D} = \{|z|<1\}$, we can 
follow $F$ on a  D-component 
by a quasiconformal map $\rho:{\mathbb D} 
\to {\mathbb D} $ so that $\rho(0)=w$,  $\rho$ is the 
identity on $\{|z|=1\}$, and $\rho$ is conformal 
on $\{|z|< 1/2\}$. The QC constant depends only on 
$|w|$ and blows up as $|w|\nearrow 1$.  Then $\rho
\circ F$ is uniformly quasiconformal (if $|w|$ is 
uniformly bounded away from $1$), and has dilatation 
supported in $T(r)$ for a uniformly bounded $r$ 
(the maximum of $r$ from Theorem \ref{QC folding thm}
 and $r_1$ above). 
 The choice of $w$ can be different 
for each D-component, but $|w|$ must be uniformly bounded
below $1$ to get a uniform quasiconformal estimate. Thus 
we have:

\begin{cor}
Suppose notation is as above. Given $0 < s< 1$ 
 there are $r< \infty$ (depending only on the bounded
geometry constants of $G$),  $K < \infty$ (depending 
on $s$ and the bounded geometry constants of $G$) and 
a $K$-quasiregular $g$ on ${\mathbb C}$ so that 
\begin{enumerate}
\item 
$g=\exp \circ \tau$ off $T(r)$.
\item 
The center of any D-component with $2n$ boundary 
vertices  is a critical point of order $n$  and 
each critical value can be specified in 
$\{|z|<s \}$.
\item 
 The only other  singular values 
of $g$ are $\pm 1$, and the corresponding 
critical points occur at vertices of $G'$.
\item 
The only asymptotic value is $\infty$, taken in 
the R-components.
\end{enumerate}
\end{cor}

We now adapt the above result to give meromorphic 
functions. Suppose we have a graph $G$, as above, 
but now the bounded  components (which we still 
assume are disks) are labeled either as D-components
or ID-components (ID for ``inverted disk'') and the 
unbounded components are labeled either as R-components 
or IR-components (IR for ``inverted R-component''). We emphasize that the new terminology ID-component (respectively, IR-component) is introduced only so as to allow a binary labelling of bounded (respectively, unbounded) components of $\mathbb{C}\setminus G$. This enables us, in what follows, to define a quasiregular function $H$ in $\mathbb{C}\setminus G'$ so that the definition of $H$ in a given component of $\mathbb{C}\setminus G'$ depends on whether that component has been labelled ``inverted'' or not. We assume that we are given such a labelling so that: 

% refers only to a labeling of components of $\mathbb{C}\setminus G$: the need for the distinction between an ID, D-components and IR, R-components and the associated terminology will become clear in what follows. }

%{\color{red} We emphasize that the new terminology ID-component and IR-component refers only to a labeling of components of $\mathbb{C}\setminus G$: the need for the distinction between an ID, D-components and IR, R-components and the associated terminology will become clear in what follows. }

\begin{itemize}
\item[(i)] D-components share edges only with R-components,
\item[(ii)]  ID-components share edges only with IR-components,
\item[(iii)] R-components may share edges with D, R or IR-components,
\item[(iv)] IR-components may share edges with ID, R or IR-components.
\end{itemize}

Apply the Folding Theorem to this graph with ID-components
momentarily
considered as D-components and IR-components considered
as R-components. Obtain the graph extension $G'$ of 
$G$ and the corresponding subdomains of the unbounded 
components. Each of these  is a subset of a R-component
or IR-component and they will be called the new R-components
and new IR-components.

Next define a function $H$   to be equal to  $F$ 
on the D and new R-components, and only use the QC-map 
$\sigma$ to modify $F$ on edges  with both sides 
belonging to  new R-components (possibly the same component). 
On the ID and new IR-components we set $H=1/F$, where we only 
use $\sigma$ to modify $F$  on edges with both sides
belonging to new IR-components (possibly the same). Note that this 
creates poles, but considered as a map into the sphere
$H$ is continuous across all edges, except possibly 
those shared by a new R-component
 and new IR-component. However,  for a point 
on such an edge, the two possible images of  $F$ are conjugate
points on the unit circle and since $1/z = \overline{z}$ on 
the unit circle, $H$ also extends continuously across such edges.
 
\begin{thm}\label{folding_with_IR}
With the assumptions above, and taking $0<s< 1$, 
 there are $r, K< \infty$ (depending 
only on the bounded geometry constants of $G$; $K$ also 
may depend on $s$) and 
a $K$-quasiregular map $g: {\mathbb C} \to \widehat {\mathbb C}$
that equals $H$ off $T(r)$. Moreover, 
\begin{enumerate}
\item 
Each IR-component contains a curve tending to $\infty$ 
along which $g$ tends to zero; thus each such 
component contributes an asymptotic value of $0$, which may be perturbed with the map $\rho$.
\item 
There are   $n$ poles 
(counted with multiplicity) in each 
ID-component that  has $2n$ vertices on  its boundary.
\item 
The  critical values corresponding to D-components
may be specified  independently in $\{|w|< s\}$
 and the  critical values 
corresponding to  ID-components may be specified 
independently in $\{|w|> 1/s\}$.  
\end{enumerate}
\end{thm}

\section{Constructing the Graph}\label{Graph_I}

In this Section we build the graph $G$ that we 
use in the proof of Theorem \ref{main}.

\begin{lem} \label{geodesic lemma}
Given $\delta >0$ and  an infinite, discrete set of points $\{z_n\}$
in the plane, 
we can construct an unbounded Jordan domain $W$  
so that 
\begin{enumerate}[{(1)}]
\item $\{z_n\} \subset W$. 
\item The points $\{z_n\}$  are all at least unit distance apart in the 
    hyperbolic metric for $W$. 
\item Every point of $\{ z_n\}$ lies within a uniformly bounded
   hyperbolic  distance 
   of some fixed hyperbolic geodesic,  $\gamma$, for $W$ that 
    connects some finite boundary point $x$ of $W$ to $\infty$.
\item  $\area(W)< \delta $ and for all $n \in \N$,
	${\rm{area}}(W \cap \{|z|>n\}) \leq \delta \exp(-n)$.
\item Every point of the plane lies within distance $1$ of $W$.
\item We can add vertices to $\partial W$ to make it into a 
uniformly analytic tree  with uniformly bounded constants. 
\item 
Each edge  $J_j$  of this tree 
is on the boundary of a region $R_j \subset W$ 
so that $\area(R_j) \simeq \diam(J_j)^2$ and 
the $\{R_j\}$ are pairwise disjoint.
\item For each edge $J_j$ of this tree, the path distance 
 in $W$ from $J_j$ to the arc of $\partial W \setminus x$
($x$ is as in part (3))
that is disjoint from $J_j$ is comparable to $\diam(J_j)$.
\end{enumerate} 
\end{lem}

%we define $W$ to be the union of small, disjoint 
%disks  centered at each point $\{z_n\}$ and connect the 
%disks, in order, by thin polygonal tubes which enter and 
%leave each disk at antipodal points of the boundary circle. }

\begin{proof} 
The proof is simple and  we only 
sketch the construction, leaving some  details 
for the reader. The main idea is illustrated 
in  Figure \ref{geodesic fig}: we take $W$ to be the union of small, disjoint 
disks $D_n$ centered at the points $\{z_n\}$ together with thin polygonal tubes connecting the 
disks, in order, which 
leave each disk at antipodal points of the boundary circle.
If the connecting tubes are thin compared to the disks, 
then the points $\{z_n\}$ are far apart in the hyperbolic 
metric, so (2) holds. We now show that (3) follows by imposing an upper bound on the relative entering width of the connecting tubes (described above). Let $x\in\partial W$, $\gamma$ a geodesic in $W$ connecting $x$ to $\infty$, and denote harmonic measure by $\omega$. Let $I_1, I_2$ denote the two components of $\partial W\setminus\{x\}$. Property (3) will follow if we can show that \[ \omega(z_n, I_1, W) > c \textrm{ and } \omega(z_n, I_2, W) > c \] for all $n$ and some $c>0$ independent of $n$. By monotonicity properties of harmonic measure we have that \[ \omega(z_n, I_1, W) \geq \omega(z_n, I_1\cap\partial D_n, D_n), \] and $I_1\cap\partial D_n$ contains a circular arc of harmonic measure (in $D_n$) uniformly bounded away from zero given an upper bound on the relative entering width of the connecting tubes. Similar considerations yield the lower bound for $\omega(z_n, I_2, W)$. 

%In $\mathbb{D}$, a geodesic connecting two boundary points $x$, $y$ comes within a bounded hyperbolic distance of a point $z$ if the two boundary arcs defined by its endpoints have harmonic measures bounded away from zero. 

\begin{figure}[htp]
\centerline{
\includegraphics[height=2.0in]{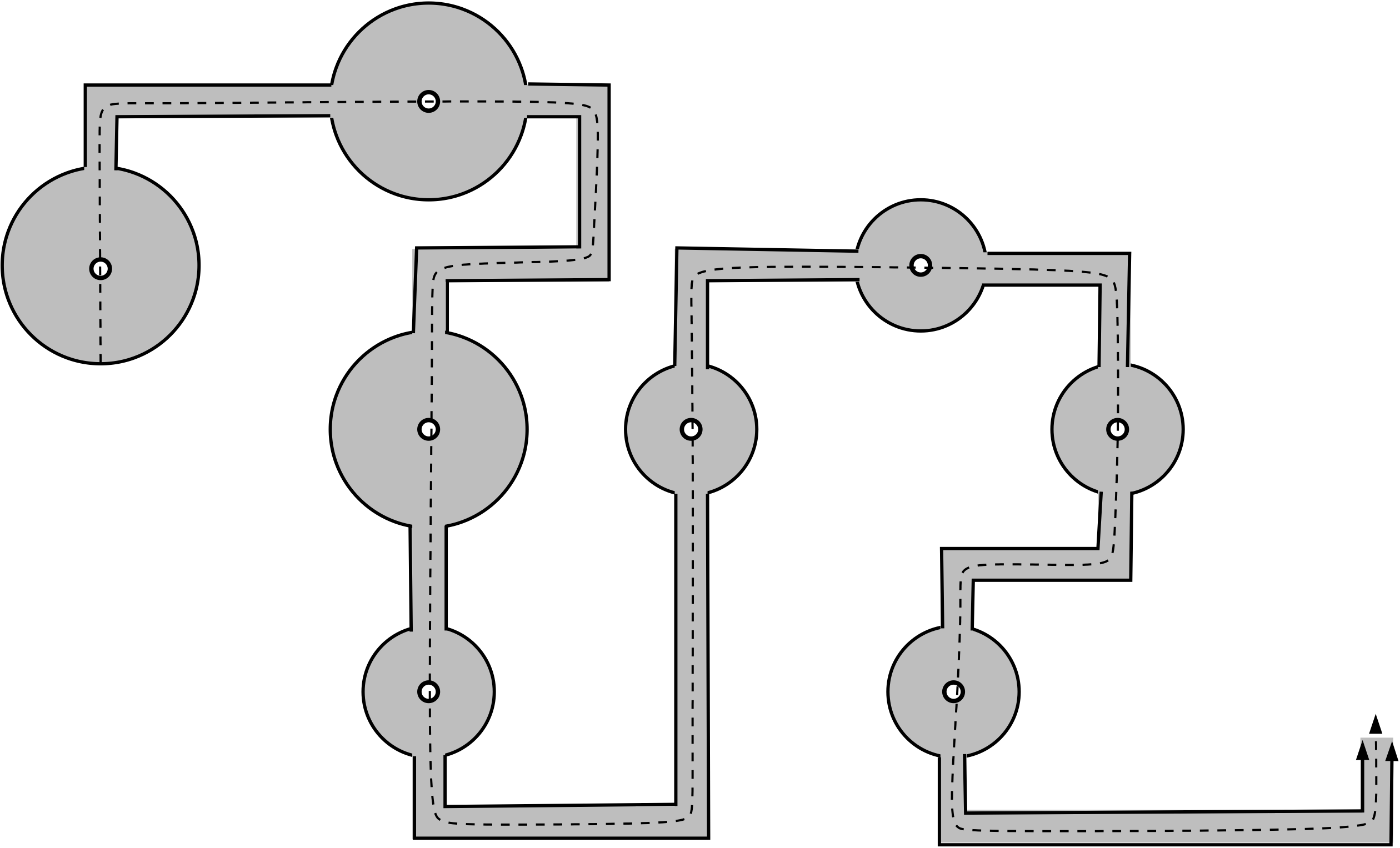}
}
  \caption{  By using small disks around each point 
  and thin corridors that enter and leave on opposite
  sides of the disks, we can build a Jordan domain 
  that satisfies Lemma \ref{geodesic lemma}.
  }
  \label{geodesic fig}
\end{figure}
Part (4) can be obtained simply  by taking the tubes and disks
in the construction small enough.  
To get (5), we can add points to $\{z_n\}$ until this 
set is $1$-dense in the plane, e.g., add any point of 
$\Z + i \Z$ that does not already have a 
point of $\{z_n\}$ within distance $1/10$ of it.
(6)  is also easy to verify:
on the tubes,  take approximately 
 evenly spaced points  where the spacing is 
comparable to the width, and partition the circles 
in a way that interpolates between the sizes of
the two tube openings. If we take a disk centered 
at the midpoint of $J_j$, whose radius is a small 
multiple of $\diam(J_j)$ (depending only on the 
bounded geometry constants), then the intersection $R_j$ of 
this disk with $W$ satisfies (7).
See Figure \ref{TubeDetail}.  We can vary the width of 
a tube (also illustrated in Figure \ref{TubeDetail}) so 
that all the previous conditions still hold, and 
the width of a tube when it enters and leaves
a disk is always comparable to the width of that 
disk; thus only a uniformly bounded number (independent 
of the disk)  of vertices is
needed on the boundary of each disk and this implies 
(8) holds.

% Part (3) uses the fact
%that a geodesic segment for the unit disk
%that starts and ends near 
%opposite points of one of the  circles must pass near the
%center of that circle. If this did not hold, then  passing to 
%a limit as the width of the tubes shrink with respect 
%to the diameters of the disks
% would give a hyperbolic  geodesic for the 
%disk connecting opposite points of the circle, but 
%that was not a diameter of the disk; a contradiction.

\end{proof}

\begin{figure}[htp]
\centerline{
\includegraphics[height=1.5in]{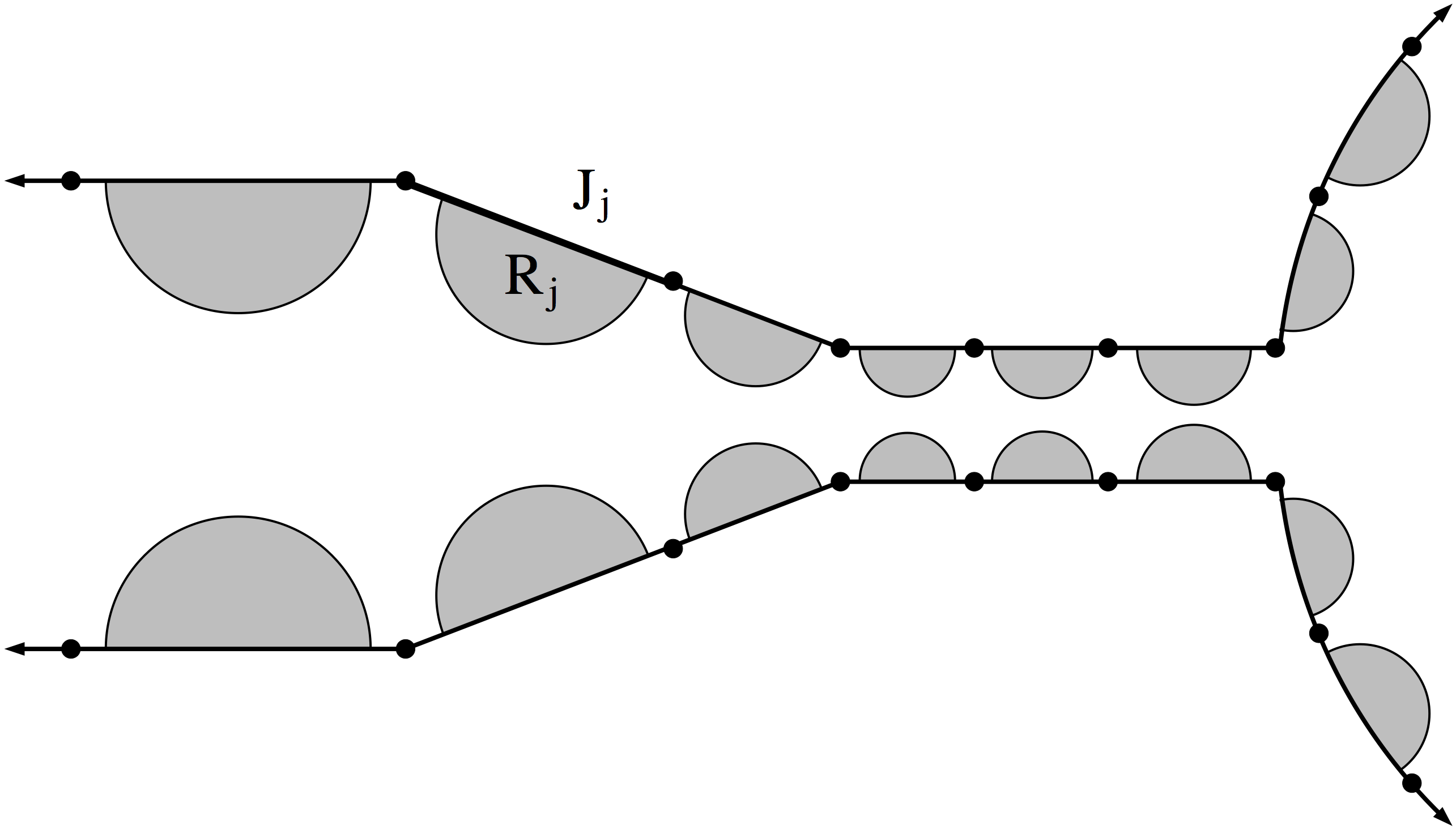}
}
  \caption{   A detail of the tube showing how the
	tube width can vary, and how to associate 
	a region $R_j \subset W$ to each side $J_j$ of $W$
	so that $\area(R_j) \simeq \diam(J_j)^2$.
  }
  \label{TubeDetail}
\end{figure}

If we conformally map $W$ to the upper half-plane, 
we can arrange for the geodesic $\gamma$ of Lemma \ref{geodesic lemma} to map 
to the positive imaginary axis (by mapping the boundary 
point $x$ in the  lemma to the origin;  post-composing 
any  conformal map from $W$ to the half-plane taking 
$\infty$ to $\infty$ 
with a translation will accomplish this). Then the 
points $\{z_n\}$ map to points in a vertical cone 
with its vertex at the origin. See Figure \ref{half-plane 1}.
Moreover, we can rescale so that the first point has 
height $1$ above the real axis.

A small disk in $W$ around each $z_n$ (say with radius 
one tenth the distance to the boundary)  will map to 
a near-circular region in the upper half-space 
and it is easy to connect the near-disks to each 
other and to $\infty$  to form a bounded 
geometry graph $\tilde{G}$ as shown in Figure \ref{half-plane 1}. We note that,
 instead of a bounded component containing the image of $z_n$, we
may also place a vertex at the image of $z_n$,
also as shown
in Figure \ref{half-plane 1}. 
The unbounded components are all approximately horizontal 
half-strips, and it is easy to verify the $\tau$-condition 
for them. 
Note that given any labeling of the bounded components 
by D or ID, we can easily label the unbounded 
components with R or IR to satisfy the necessary 
adjacency restrictions.

\begin{figure}[htp]
\centerline{
\includegraphics[height=2.5in]{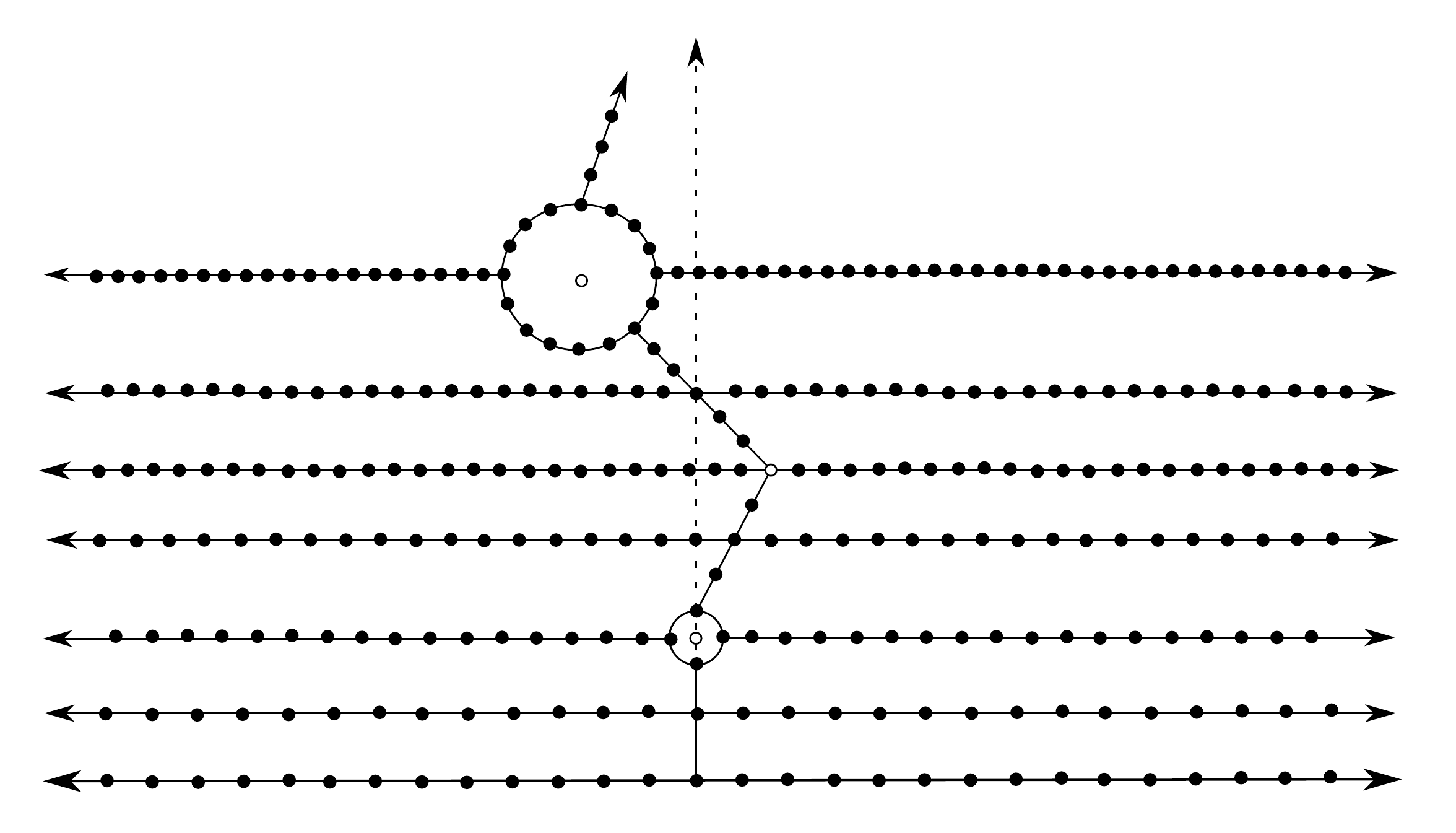}
}
  \caption{ Building a  bounded geometry graph with either a  bounded 
     component centered at each point or vertex there. } 
  \label{half-plane 1}
\end{figure}

The bounded geometry and $\tau$-lower bound are 
clear for all the complementary components of $\tilde{G}$,
except possibly the two components that border
the real line. These require a separate argument; we 
want to show the vertices on the real line can 
be taken with all spacings $\simeq 1$. 
By the  bounded geometry of $\partial W$, the 
vertices on $\partial W$ map to points on the real line 
that define a quasisymmetric partition of the line (see Lemma \ref{Lemma_4.1}).
Condition (8) of Lemma \ref{geodesic lemma} implies that
the spacing between the vertices on $\R$ grows 
exponentially (this is precisely Lemma 8.1 of 
\cite{MR3653246}) and hence the spacing is bounded below. 
Thus by adding more points to $\tilde{G}$ along the real axis,
if necessary, we can assume  
every  edge on the real axis  has length at most $1/4$ and 
without changing the bounded geometry constants of $\tilde{G}$.
This verifies the bounded geometry condition and 
lower $\tau$-bound for the two components that 
border the real line.
Moreover, adding the corresponding vertices to $\partial W$
does not increase the bounded geometry constant or 
the uniformly analytic constant of $\partial W$. 

Therefore, Lemma 4.1 of \cite{MR3653246} implies 
that the  image $G$ of $\tilde{G}$
under the conformal map back to $W$ will be a uniformly
analytic graph $G$ contained in $\overline{W}$.
 The $\tau$-condition will be 
automatically satisfied for components inside $W$, since
this is a conformally invariant condition. 

These account for all the complementary components of 
$G$ except  for one: the  complement $V$  of $\overline{W}$.
This is also an unbounded Jordan domain, 
but it is not clear whether  it satisfies the  $\tau$-condition.
 However, there is a very simple trick for fixing this that 
we take from \cite{MR3653246}. 
Let $\varphi:V \to \uhp$ be the conformal map of $V$ to 
the upper half-plane, taking infinity to infinity. We let 
$\Phi: \uhp \to V$ be its inverse. The vertices 
on $G$  on $\partial W = \partial V$ map to points 
on the real line. By Lemma \ref{Lemma_4.1}, 
these points define a quasisymmetric partition of the real line.
We define a graph in 
the closed upper half-plane by adjoining to the real line 
vertical rays, and placing evenly spaced vertices on each 
ray, where the spacing is the minimum distance of that
ray to its two neighbors to the left and right. This 
defines an infinite ``comb'' tree.  
See Figure \ref{half-plane 2}.

\begin{figure}[htp]
\centerline{
\includegraphics[height=1.5in]{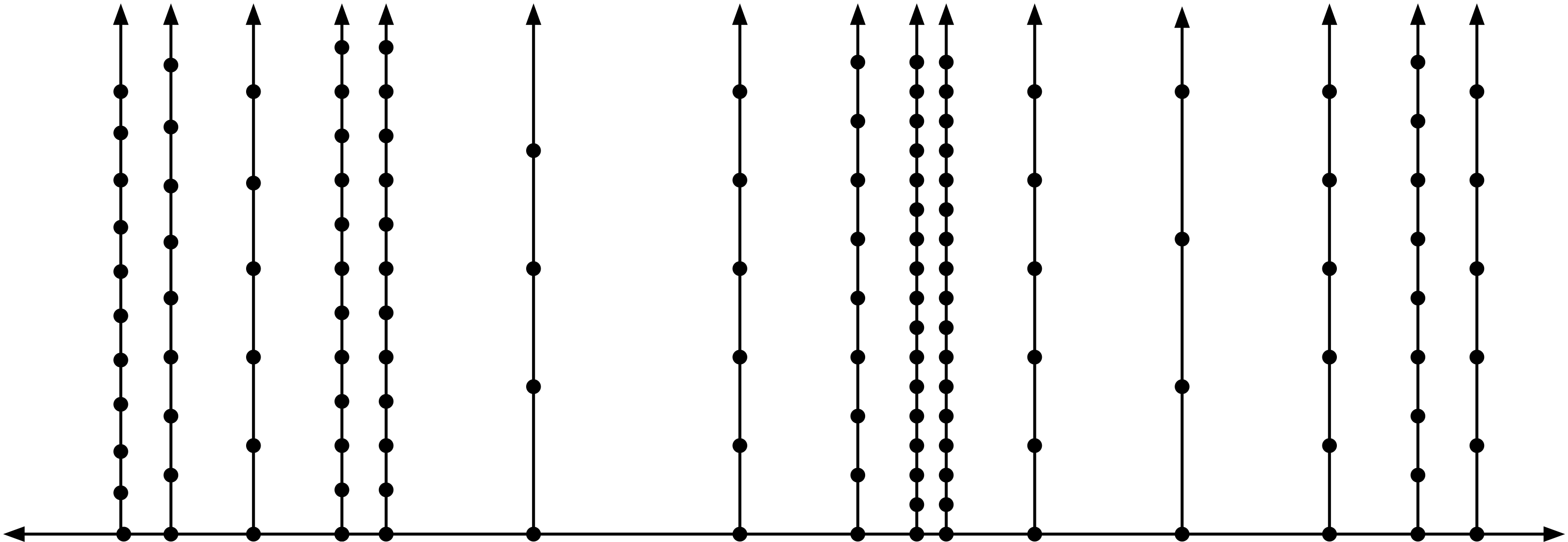}
}
  \caption{Any quasisymmetric partition of the real 
    line can be extended to a bounded geometry graph
    in the upper half-plane that satisfies the 
    $\tau$-condition. } 
  \label{half-plane 2}
\end{figure}

 By Lemma 6.1 of 
\cite{MR3653246},   this tree is uniformly analytic 
and every component 
satisfies the $\tau$-condition for an appropriate choice 
of $\tau$. Therefore by Lemma 4.1 of \cite{MR3653246} again,
the same is true for the conformal image of this 
graph in $V$. Adding this image to $G$ gives a new 
uniformly analytic tree (which we will still call $G$).
We mark all the new components (i.e., the subdomains of $V$)
as R-components. By construction, these only share edges 
with the two unbounded  sub-domains of $W$ that border 
$\partial W$, hence they do not share edges  with any 
ID-component, as required in condition (iii) in the discussion preceding Theorem \ref{folding_with_IR}. 

In fact, the $\tau$-condition remains valid for 
the infinite comb tree  even if the 
spacing of the points decays exponentially in the height.
Therefore we can place the vertices so that the graph 
has bounded geometry, the $\tau$-condition holds on 
each vertical half-strip, and 
the area of $T(r)$ intersected with any of the half-strips 
decays exponentially as we move away from the boundary 
of the half-plane.
Next we use the distortion theorem for conformal maps  
to prove an analogous estimate 
for the conformal image of this graph inside $V$.

\begin{lem}
	\label{area estimate}
Suppose the domain $W$ and graph $G$ are as   described above,
and that $g$ is the corresponding quasi-regular 
function given by the folding construction. 
Let $E$ be the set where $g$ is not holomorphic
(note that $E$ is contained in $T(r)$ by construction).
There is a $\alpha >0$ so that for any $\delta >0$,
 we can choose $W$, $G$ and $g$ so that  
\[ \area(E \cap \{|z|\geq n\}) \leq 
\delta \cdot \exp(-\alpha n), \quad n=0,1,2,\dots \] 
\end{lem} 

\begin{proof}
The domain $W$ was chosen so that $\area(E \cap W)$
satisfies this estimate, so we only need to 
worry about  $\area(E \cap V) = \area(E \setminus
\overline{W})$.
	
We know that 
$E$ will be contained in the conformal image of 
the set $T(r)$  corresponding to the comb tree 
in the upper half-plane  illustrated in 
Figure \ref{half-plane 2}. % figure 5
Recall  that $\Phi:\mathbb{H} \to V$ is a conformal map.
The tree in the upper half-plane 
consists of  vertical   rays that define 
vertical half-strips
$\{S_j\}$  with the finite edge $I_j$ lying on 
the real axis. These edges correspond to the 
edges $\{J_j\}$  on $\partial W$ via $\Phi$. 
In Figure \ref{half-plane 2}, the points 
on each vertical ray are shown as being evenly 
spaced, with the spacing being comparable to 
the distance from the ray to its two neighboring 
rays. However, we can space the points at height 
$y$ so they are only separated by distance
\begin{eqnarray}\label{decay rate}
 \simeq |I_j| \exp( -c y/|I_j|),
\end{eqnarray}
 for some $c>0$ and still have the 
$\tau$-condition. This holds since the conformal 
map from a half-strip to a half-plane is given 
in terms of the $\sinh$ function, which has 
exponential growth in the half-strip (for a
single strip we could take $c = \pi$, but since 
the spacing  on a ray depends on the width of both 
adjacent strips, we use a positive $c$
that depends on the relative sizes of adjacent
$I_j$'s).
See Figure \ref{Half-Strip}; note that  
the vertical half-strip
is drawn horizontally to make the illustration 
clearer.

\begin{figure}[htp]
\centerline{
\includegraphics[height=1.5in]{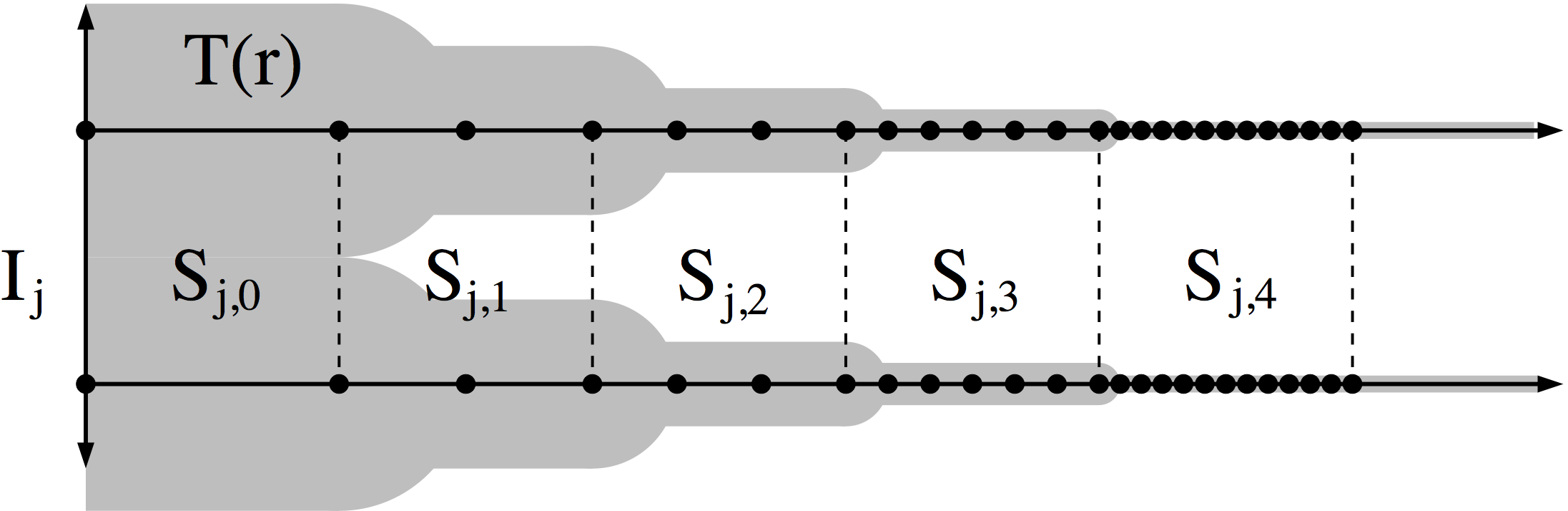}
}
  \caption{The  half-strip $S_j$ is cut into squares
	 $\{ S_{j,k}\}$ whose intersection 
	with $T(r)$ (shaded) has Euclidean area that decays 
	exponentially with $k$. The picture has been 
	rotated by $90^\circ$ compared to Figure 
	\ref{half-plane 2}.
	}
  \label{Half-Strip}
\end{figure}

Now cut $S_j$ into  disjoint squares 
$\{S_{j,k}\}_{k=0}^\infty$ of side length $|I_j|$, 
where $S_{j,k}$  denotes the square whose 
Euclidean distance from the boundary segment $I_j$
is  $k|I_j|$.  Because 
of the exponential decrease in the 
spacing between vertices, the fraction 
of this square that  hits $T(r)$ is bounded 
by $O(\exp(-ck))$ ($c>0$ as in (\ref{decay rate})).
We will show that a similar estimate holds, even 
after we map these squares back to the region $V$:

\begin{lem}\label{exp est}
	With notation as above, 
 $$
\area( \Phi(T(r) \cap S_{j,k})) \leq C \diam(J_j)^2 
 \cdot \exp(-ck/2) ,$$ 
for $k \geq 0$, where $C < \infty$ is fixed and 
	$c >0$ is as in (\ref{decay rate}).
\end{lem}

\begin{proof} 
The case $k=0$ (the square that is adjacent to 
the boundary of the half-plane) is  
different from the cases $k \geq 1$, and we 
deal with it first.

Recall that $\varphi: V \to \uhp$ is our choice of
conformal map and we let $\Phi: \uhp \to V$ denote 
its inverse.
In the case $k=0$,  we simply bound the area
of $\Phi(T(r) \cap S_{j,0})$ by the area
of $\Phi(S_{j,0})$ (i.e., we assume $T(r)$ fills
the entire square)   and we  claim the 
latter set has diameter  bounded by a 
uniform multiple of  $\diam(J_j)$.

Let $x_j$ be the center of the boundary segment 
$I_j$,  let $y_j = |I_j|$ and define $z_j = x_j +i y_j
\in \uhp$ and $w_j = \Phi(z_j) \in V$.
 By Koebe's theorem 
\begin{equation}\label{Koebe_estimate}\dist(w_j, \partial V) \simeq y_j |\Phi'(z_j)|,\end{equation}
and hence 
\begin{eqnarray}\label{exercise 8}
  y_j |\Phi'(z_j)| = O( \diam(J_j)),
 \end{eqnarray}
e.g., see Exercise IV.8 in \cite{MR2150803}.

%\begin{figure}
%\centering
%\scalebox{.6}{\input{referee_argument_converted.png}}
%\scalebox{.6}{\input{main_theorem_converted.png}}
%\scalebox{.6}{\input{main_theorem_strategy_larger2.png}}
%\caption{ Illustrated is the notation used in the proof of Lemma \ref{exp est} for the case $k=0$.}
%\label{fig:referee_argument}       % Give a unique label
%\end{figure}

\begin{figure}[htp]
\centerline{
\includegraphics[height=2.3in]{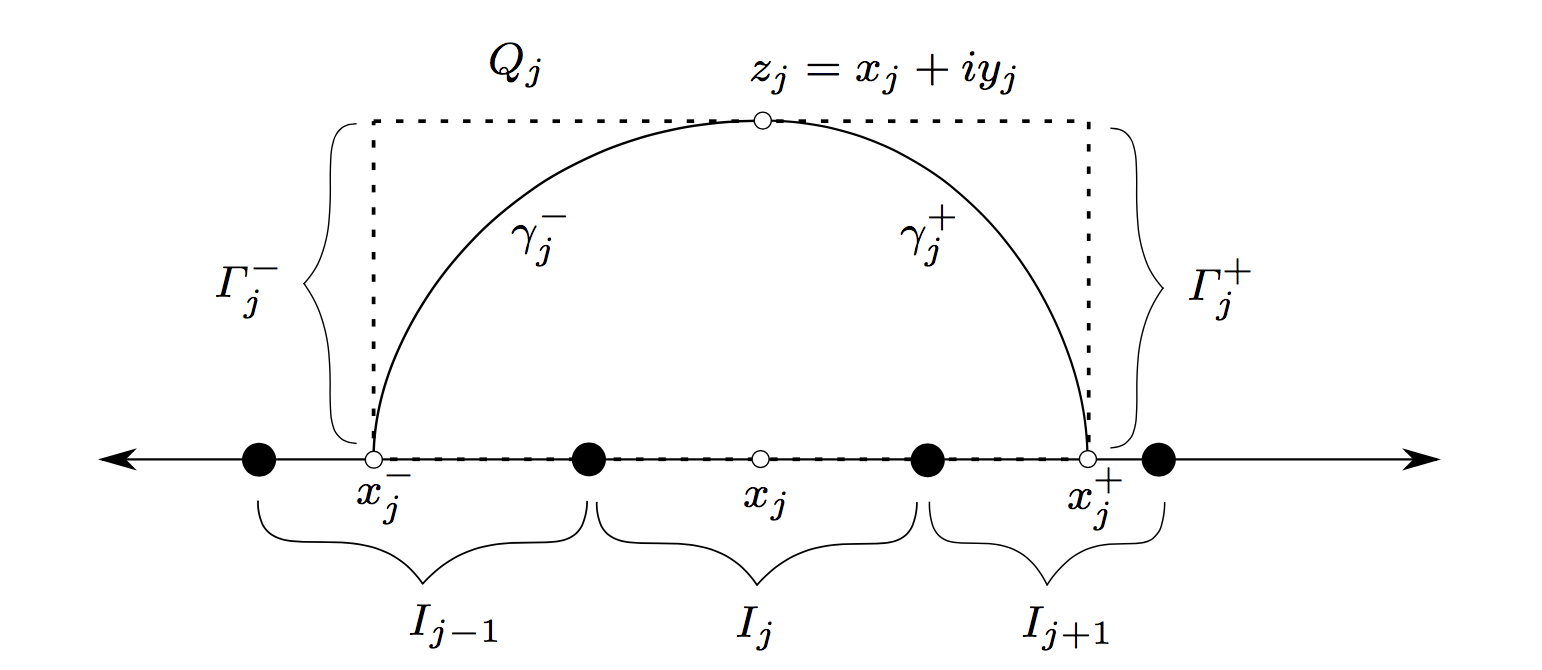}
}
  \caption{  Illustrated is the notation used in the proof of Lemma \ref{exp est} for the case $k=0$.  }
  \label{fig:referee_argument}     
\end{figure}

%Similarly, there is a point $x_j^+\in I_{j+1}$ such that the geodesic $\gamma_j^+$ connecting $z_j$ to $x_j^+$ satisfies an analogous estimate.

Assume the boundary intervals $\{I_j\}$  are 
numbered consecutively, so that $I_{j-1}$ 
and $I_{j+1}$ are adjacent to $I_j$.  
By Lemma \ref{Lemma_4.1},  
 the intervals $I_{j-1}$, $I_j$, $I_{j+1}$ have  uniformly comparable lengths (uniform over $j$). Thus by Corollary 4.18 of \cite{MR1217706} and (\ref{Koebe_estimate}), there exists a point $x_j^-\in I_{j-1}$ such that the geodesic $\gamma_j^-$ in $\mathbb{H}_u$ with endpoints $x_j^{-}$ and $z_j$ satisfies \begin{equation}\label{geodesic_length} \textrm{length}(\Phi(\gamma_j^-)) \leq My_j\left|\Phi'(z_j)\right|, \end{equation} for $M<\infty$ independent of $j$ (see Figure \ref{fig:referee_argument}). Let $\Gamma_j^-$ denote the vertical segment connecting $x_j^-$  to $x_j^-+iy_j$. We claim that $\Phi(\Gamma_j^-)$ and $\Phi(\gamma_j^-)$ have uniformly comparable lengths (uniform over $j$). To see this, we cut $\Gamma_j^-$ and $\gamma_j^-$ into subsegments $\Gamma_{j,k}^-$ and $\gamma_{j,k}^-$ for $k=0$, $1$, $2$, $\dots$, where the $k^{\textrm{th}}$ subsegment is defined to be the subsegment lying in the horizontal strip \[\{x+iy:2^{-k-1}y_j \leq y \leq 2^{-k}y_j\}. \] Two such subsegments $\Gamma_{j,k}^-$, $\gamma_{j,k}^-$ lie in a common hyperbolic disc of fixed radius (independent of $j$, $k$), and so by Koebe's distortion theorem the derivative of $\Phi$ is uniformly comparable throughout this disc. Thus since the lengths of $\Gamma_{j,k}^-$, $\gamma_{j,k}^-$ are comparable, \[\textrm{length}(\Phi(\Gamma_{j,k}^-)) \textrm{ and }  \textrm{length}(\Phi(\gamma_{j,k}^-)) \] are uniformly comparable (uniform over $j$, $k$). Thus the bound (\ref{geodesic_length}) implies that \begin{equation}\label{comparable_length} \textrm{length}(\Phi(\Gamma_{j}^-)) \leq M' y_j\left|\Phi'(z_j)\right| \end{equation} for $M'<\infty$ independent of $j$. Analogously, there is a point $x_j^+\in I_{j+1}$ such that the image under $\Phi$ of the vertical segment $\Gamma_j^+$ connecting $x_j^+$ to $x_j^++iy_j$ satisfies a similar bound. Lastly, consider the horizontal segment $l_j$ connecting $x_j^-+iy_j$ to $x_j^++iy_j$. Since $I_{j-1}$, $I_j$, $I_{j+1}$ have uniformly comparable lengths, \begin{equation} \textrm{length}(\Phi(l_j)) = O(y_j\left|\Phi'(z_j)\right|) \end{equation} by Koebe's distortion theorem. Let $Q_j$ be the Euclidean rectangle with vertices $x_j^-$, $x_j^+$, $x_j^++iy_j$, $x_j^-+iy_j$. Summarizing, we have $\Phi(S_{j,0}) \subset \Phi(Q_j)$, % the intervals $I_{j-1}$, $I_j$, $I_{j+1}$ have  {\color{red} uniformly} comparable lengths (uniform over $j$). Thus by Corollary 4.18 of \cite{MR1217706} and (\ref{Koebe_estimate}), there exists a point $x_j^-\in I_{j-1}$ such that the geodesic $\gamma_j^-$ in $\mathbb{H}_u$ with endpoints $x_j^{-}$ and $z_j$ satisfies \begin{equation}\label{geodesic_length} \textrm{length}(\Phi(\gamma_j^-)) \leq My_j\left|\Phi'(z_j)\right|, \end{equation} for $M<\infty$ independent of $j$ (see Figure \ref{fig:referee_argument}). Similarly, we may define $x_j^+\in I_{j+1}$ such that for the geodesic $\gamma_j^+$ connecting $z_j$ to $x_j^+$, we have \begin{equation}\label{geodesic_length2} \textrm{length}(\Phi(\gamma_j^+)) \leq My_j\left|\Phi'(z_j)\right|. \end{equation} Let $Q_j$ be the Euclidean rectangle with vertices $x_j^-$, $x_j^+$, $x_j^++iy_j$, $x_j^-+iy_j$. Let $\Gamma_j^-$ denote the vertical segment connecting $x_j^-$  to $x_j^-+iy_j$. We claim that $\Phi(\Gamma_j^-)$ and $\Phi(\gamma_j^-)$ have uniformly comparable lengths (uniform over $j$). To see this, we cut $\Gamma_j^-$ and $\gamma_j^-$ into subsegments $\Gamma_{j,k}^-$ and $\gamma_{j,k}^-$ where they cross the horizontal lines $y=2^{-k}y_j$ for $k=0$, $1$, $2$, $\dots$. Two such subsegments $\Gamma_{j,k}^-$, $\gamma_{j,k}^-$ lie in a common hyperbolic disc of fixed radius (independent of $k$), and so by Koebe's distortion theorem the derivative of $\Phi$ is uniformly comparable throughout this disc. Thus \[\textrm{length}(\Phi(\Gamma_{j,k}^-)) \textrm{ and }  \textrm{length}(\Phi(\gamma_{j,k}^-)) \] are uniformly comparable (note the lengths of $\Gamma_{j,k}^-$, $\gamma_{j,k}^-$ are comparable). Thus the bound (\ref{geodesic_length}) implies that \begin{equation}\label{comparable_length} \textrm{length}(\Phi(\Gamma_{j}^-)) \leq M' y_j\left|\Phi'(z_j)\right| \end{equation} where $M'$ is independent of $j$. Analogously, the image under $\Phi$ of the vertical segment $\Gamma_j^+$ connecting $x_j^+$ to $x_j^++iy_j$ satisfies the bound (\ref{comparable_length}). Lastly, since $I_{j-1}$, $I_j$, $I_{j+1}$ have uniformly comparable lengths, the length of the horizontal segment of $Q_j$ connecting $x_j^-+iy_j$ to $x_j^++iy_j$ is $O(y_j\left|\Phi'(z_j)\right|)$ by Koebe's distortion theorem. Summarizing, we have $\Phi(S_{j,0}) \subset \Phi(Q_y)$,
and the boundary of the latter set is contained
inside the union of  $\Phi(\Gamma_j^-)$, $\Phi(\Gamma_j^+)$, $\Phi(l_j)$, 
$J_{j-1}$, $J_j$,  and $J_{j+1}$; all of
these have diameter $O(\diam(J_j))$ by (\ref{exercise 8}).
Thus 
$\diam(\Phi(S_{j,0})) = O(\diam(J_j))$, so the $k=0$ case of Lemma \ref{exp est} has been proved.

Next  we verify Lemma \ref{exp est} for $k \geq 1$.
In this case, $S_{j,k}$ has bounded hyperbolic 
diameter, so  Koebe's distortion theorem implies 
\begin{eqnarray}\label{ratio}
\frac { \area(\Phi(T(r) \cap S_{j,k}))}
      { \area ( \Phi( S_{j,k}))} 
\simeq 
\frac { \area(T(r) \cap S_{j,k})}
      { \area (  S_{j,k})} 
\simeq \exp(-ck)\textrm{,}
\end{eqnarray}
where $c>0 $ is as in (\ref{decay rate}).
Thus it suffices to bound $ \area(\Phi( S_{j,k}))$.
To do this,  
we  use Lemma 16.1 of \cite{Albrecht-Bishop}:

\begin{prop}\label{lem:diameters of disks}
Let $\Omega\neq \mathbb{C}$ be simply connected and let
 $\varphi:\Omega \to \uhp$ be a conformal map to 
the  upper  half-plane.  Let $\Phi:\uhp \to \Omega$
 denote the inverse of $\varphi$. Let
$w = x+it$ and $z=x+iy$ with $y>t$ and let $X\subset\uhp$
be a simply connected neigbourhood
 of $z$ with hyperbolic radius bounded by $r$.
Then 
\[
\diam (\Phi(X))=O( |\Phi'(w)| \frac yt \diam (X))
\]
where the constant depends only on $r$.
\end{prop}

The statement of  this in \cite{Albrecht-Bishop}
is for the special case  $t=1$ and a map into the 
right half-plane, but the version above
 follows immediately 
by considering our $\varphi$  composed 
with the linear map $z \to -\frac it  z$. 
The proof given in  \cite{Albrecht-Bishop}
is  a short deduction from the standard 
distortion theorem for conformal maps, 
e.g., Theorem I.4.5 of \cite{MR2150803}.

We apply Proposition \ref{lem:diameters of disks}
using $\varphi$,   $t= |I_j| = y_j$, 
$w = x_j +iy_j = z_j$ and  $X=S_{j,k}$, $k \geq 1$. 
Note that $y/t =k$ and $\diam(X) = y_j$. Then 
$$\diam(\Phi(X))  = O(|\Phi'(z_j)| \cdot k 
\cdot y_j ),$$
and using (\ref{exercise 8}) gives
$$\diam(\Phi(X))    = O(\diam(J_j) \cdot k).$$ 
Since $k = o(\exp(ck/4))$, we get 
$$\area(\Phi(X))    = O(\diam(J_j)^2 \exp(ck/2 )), $$ 
and hence, using (\ref{ratio}),  
$$\area( E \cap \Phi(X))  = O(\diam(J_j)^2 \exp(-ck/2 )), $$ 
which gives the  $k \geq 1$  cases of Lemma \ref{exp est}.

\end{proof}

Now that  Lemma \ref{exp est} is established, we can 
finish the proof of Lemma \ref{area estimate}. 
Let $E_j = E \cap \Phi(S_j)$ and note that 
$$ \area(E_j) =\sum_{k=0}^\infty \area (E_j \cap \Phi(S_{j,k})) 
 = O( \diam(J_j)^2 \sum_{k=0}^\infty \exp(-kc/2)) 
 = O( \area(R_j) ), $$
 where $\{R_j\}$ are the regions from Lemma 
 \ref{geodesic lemma}.
Next, let $U_n = \{z \in \mathbb{C} : |z|\geq n\}$. Then 
$$ \area(E \cap U_n) =\sum_j \area(E_j \cap U_n).$$
We break this sum into two parts, depending on 
whether $J_j$ is contained in $U_{n/2}$ or not. 
For the first sum, we have
$$ \sum_{j:J_j \subset U_{n/2}}  \area(E_j) 
= O(\sum \area(R_j))= O(\area(W \cap U_{n/2})) 
= O(\exp(-n/2)).$$

If $J_j$ is not contained in $U_{n/2}$, then 
let $\gamma_j$ be the hyperbolic geodesic
connecting the two endpoints of $J_j$ inside 
$V$. Note that $\diam(\gamma_j)$ is uniformly bounded from above by a theorem of Gehring and Haymann (see, for instance, Exercise III.16 in \cite{MR2150803}) since $\diam(J_j)$ is uniformly bounded from above by the construction of $G$. Thus any curve that connects
$\gamma_j$ to $\partial U_{n}$ inside $V$ has 
quasi-hyperbolic length at least comparable 
to $n$ (recall that we have constructed $W$
so that every point of $V$
is within Euclidean distance $1$ of $\partial V$).
By Koebe's distortion theorem, the hyperbolic 
and quasi-hyperbolic metrics are comparable, 
and hence the hyperbolic distance from $\gamma_j$ 
to $\partial U_{n} = \{z \in \mathbb{C} : |z|=n\}$ is also comparable to $n$.
Thus any square $S_{j,k}$ whose $\Phi$-image 
hits $U_n$ has at least hyperbolic distance 
$\simeq n$ to $S_{j,0}$ in the upper half-plane
and hence $k \geq \exp(a n)$ for some fixed $a>0$.
Now fix $j$ and sum all 
over all the squares $S_{j,k}$ whose $\Phi$-images hits $U_{n/2}$:
%(also use the fact that $k^2 = O(\exp(ck/2))$): 
\begin{eqnarray*}
 \area (E_j \cap U_n)
&\leq& 
\area(R_j) O( \sum_{k > \exp(an)}  \exp(-ck/2) )\\
&\leq& \area(R_j)  O( \exp(-a c n/2) ),
\end{eqnarray*}
for some $a >0$.
The $\{R_j\}$ are  pairwise disjoint and 
contained in $W$, so 
summing  $\area(R_j)$ over all $j$ is bounded 
by $\area(W)$. Taking $\alpha = ac/2$ 
completes the proof of Lemma \ref{area estimate}.

\end{proof} 

%\begin{cor}
% If we use the measurable Riemann mapping theorem 
%to find a quasiconformal map $\varphi$ of 
%the plane so that $f = g \circ \varphi$ is 
%holomorphic, then $\varphi$ can be chosen 
%so that it satisfies 
%$$ |\varphi(z) - z| = O(\frac {\epsilon}{1+|z|}),$$
%for all $z$ in the plane.
%\end{cor} 
%This follows immediately from a result of Dyn'kin
%in \cite{MR1466801}.

%The construction  above, combined with our earlier 
%remarks on meromorphic folding, implies:

\noindent We end this Section with the following consequence of Theorem \ref{folding_with_IR} and Lemmas \ref{geodesic lemma} and \ref{area estimate}:

\begin{lem}\label{folding_final_}
Suppose $\varepsilon, \delta >0$ and 
	suppose we are given
an infinite, discrete set of points 
$\{z_n\}_{n=1}^{\infty}$ in the plane, and a sequence
$\{w_n\}_{n=1}^{\infty}$ so that for each $n$, either $w_n=\pm1$ or $|w_n| $ is 
uniformly bounded away 
from $1$ (i.e., $||w_n|-1|> \varepsilon >0$). 
Then we can find a quasiregular map 
$g: {\mathbb C} \to \widehat {\mathbb C}$ so that:
\begin{enumerate}[{(1)}]
%\begin{enumerate}
   \item For all $n \in \N$, 
$g$ has a critical point at $z_n$ whose 
critical value is $w_n$. 
   \item $S(g) = \{w_n\}_{n=1}^{\infty} \cup \{ \pm 1\}$, i.e., 
 the only other singular values of $g$ are $\pm 1$  
(these correspond to the  critical points occurring 
at the vertices of the graph $G'$), 
and the asymptotic value $0$ coming from the 
IR-components ($0$ can be replaced by any 
value in $\{|w|<1-\varepsilon\}$).
   \item The map $g$ is conformal except on 
a set  $E$ whose area is less than $\delta$,
and is exponentially small near 
$\infty$, i.e., we have 
 ${\rm{area}}(E\cap \{|z|>n\})<  \delta \exp(-n)$.
    \item Moreover, $K$ does not depend on the 
particular critical values $\{w_n\}$ chosen, but only on 
$\{z_n\}$ and $\varepsilon$.
%   \item Moreover, $K$ and $E$ do not depend on the 
%particular critical values chosen, but only on 
%$\{z_n\}$, $\varepsilon$ and $\delta$.
\end{enumerate}
\end{lem} 

\begin{proof} Given the sequences $\{z_n\}_{n=1}^{\infty}$ and $\{w_n\}_{n=1}^{\infty}$, one obtains a bounded geometry graph $G$ through Lemma \ref{geodesic lemma} (and the ensuing discussion in Section \ref{folding_first_section}) with the following property: for each $n\geq1$, $G$ has a vertex at $z_n$ if $w_n=\pm1$, and if $w_n\not=\pm1$, $G$ has a D-component or ID-component centered at $z_n$ according to whether $|w_n|<1$ or $|w_n|>1$, respectively. The bounded geometry constants of $G$ depend only on $\{z_n\}_{n=1}^{\infty}$ by Lemmas \ref{geodesic lemma} and \ref{area estimate}. Theorem \ref{folding_with_IR} then yields (1), (2) and (4). Property (3) is a  consequence of Lemma \ref{area estimate}.

\end{proof}

\section{Reducing Theorem \ref{main} to 
a special case}\label{Reducing}

Next we show that it suffices to prove 
Theorem \ref{main} using extra hypotheses 
on the set $S$.

First,  after conjugating by a conformal linear transformation $z\rightarrow az+b$, we may assume $\pm1\in S$. In other words, given a general $S$ as in Theorem \ref{main}, we can always find a conformal linear transformation $z\rightarrow m(z)$ so that $\pm1 \in m(S)$. It suffices to find some meromorphic $f$ satisfying the conclusion of Theorem \ref{main} for $m(S)$, since then $m^{-1}\circ f\circ m$ is the desired meromorphic function in the conclusion of Theorem \ref{main} for the initial $S$.

Next, we claim that since $|S|\geq4$, we may further always choose the conjugating linear transformation $z\rightarrow m(z)$ so that there is some $s\in S$ with $|m(s)|<1$ in addition to $\pm1 \in m(S)$. Indeed, note that it would suffice to find three points $s,p,q \in S$ so that the circle whose diameter is the straight line segment $[s,p]$ joining $s, p$ contains the third point $q$ in its interior. Moreover, this happens if and only if the angle subtended by $[s,p]$ at $q$ is greater than or equal to $\pi/2$. If three points of $S$ are collinear, the statement is obvious. If we assume a point $r\in S$ is in the interior of the convex hull $T$ of $s,p,q \in S$, then the three angles subtended at $r$ by the three edges of $T$ sum to $2\pi$, hence one of the angles is greater than $\pi/2$, as needed. Lastly, if no point of $s,p,q,r$ is in the convex hull of the other three, then the two segments connecting alternating pairs of points must cross one another. If the closed disk corresponding to each segment does not contain either of the two points of the other pair, the circles must cross in at least four points, which is impossible. Hence we can always choose the conjugating linear transformation $z\rightarrow m(z)$ so that there is some $s\in S$ with $|m(s)|<1$ in addition to $\pm1 \in m(S)$.
 
%To avoid a more cumbersome statement of Theorem \ref{main} allowing for finite essential singularities, we prefer instead in this case to appeal to Theorem \ref{DKM} of \cite{2017arXiv170906866D}. 

Furthermore, we claim that we may further assume that $S$ does not intersect some open neighborhood of $|z|=1$ other than at $\pm1$. Indeed, suppose for the moment that we can prove Theorem \ref{main} for such an $S$. If we then consider the case when $S$ has finitely many points of modulus one, we adjust $S$ by moving each such point of modulus one (other than $\pm1$) by a radial distance $\varepsilon/2$ inside of $\mathbb{D}$. Then we only need to apply Theorem \ref{main} with $\varepsilon/4$ to this adjusted $S$ to obtain the desired result. 

Let us return to Theorem \ref{main}, where we are given $\varepsilon>0$, a discrete sequence $S=(s_n)$ and some dynamics $h:S \rightarrow S$. By the above discussion, we may assume henceforth that $\pm1 \in S$, there is some $s\in S$ with $|s|<1$, and that $\forall s \in S$ with $s\not=\pm1$, $||s| - 1|>2\varepsilon$. Thus given any sequence $(t_n^*)$ with $|t_n^* - h(s_n)| < \varepsilon$ for all $n$ and $t_n^*=h(s_n)$ if $h(s_n)=\pm1$, we may apply Lemma \ref{folding_final_} with $\{z_n\}=\{s_n\}$, $\{w_n\}=\{t_n^*\}$ to yield a quasiregular map $g: \mathbb{C} \rightarrow\hat{\mathbb{C}}$ such that $g$ has a critical point at each $s_n$ with corresponding critical value $t_n^*$. By the measurable Riemann mapping theorem, there exists some quasiconformal map $\phi: \mathbb{C}\rightarrow\mathbb{C}$ such that $f=g\circ\phi^{-1}$ is meromorphic. Moreover, since Lemma \ref{folding_final_} guarantees that the support of the dilatation of $\phi$ is exponentially small near $\infty$, we may normalize $\phi$ so that $\phi(z)=z+O(1/z)$ near $\infty$ (see for instance \cite{MR1466801}). Furthermore, by choosing $\delta$ in Lemma \ref{folding_final_} sufficiently small, we may guarantee that $|\phi(z)-z|<\varepsilon$ for all $z\in\mathbb{C}$. In the next section we will see how to choose $\{w_n\}=\{t_n^*\}$  so that  $\phi^{-1}(t_n^*)=h(s_n)$ for all $n$.

If we can prove that there is a choice of $\{w_n\}=\{t_n^*\}$  so that  $\phi^{-1}(t_n^*)=h(s_n)$ for all $n$, Theorem \ref{main} will be proven in the case that $h$ is onto. Indeed, we would then have $P(f)=S(f)= \{t_n^*\}$, and $\psi: S\rightarrow P(f)$ defined as $\psi(h(s_n)):=t_n^*$ (see also the discussion in Section 1). If $h$ is not onto, the definition $\psi(h(s_n)):=t_n^*$ of course does not define $\psi$ on the entirety of $S$ and we may have, for instance, that $|P(f)|=|S(f)|<|S|$. However we claim that in what follows, we may assume that $h$ is onto. Indeed, if $h$ is not onto, we may augment the sequence $S$ with auxiliary points to form a discrete sequence $\tilde{S}$, and extend $h$ to a function $\tilde{h}:\tilde{S}\rightarrow S$ such that $\tilde{h}$ is onto $S$. Then we apply Lemma \ref{folding_final_} with $\{z_n\}=\{\tilde{s}_n\}$ and $\{w_n\}=\{t_n^*\}$ where $\{t_n^*\}$ is any sequence with $|t_n^* - \tilde{h}(\tilde{s}_n)| < \varepsilon$ for all $n$ and $t_n^*=\tilde{h}(\tilde{s}_n)$ if $\tilde{h}(\tilde{s}_n)=\pm1$. Then, again, if we can choose $\{w_n\}=\{t_n^*\}$  so that  $\phi^{-1}(t_n^*)=\tilde{h}(\tilde{s}_n)$ for all $n$, then $f:=g\circ\phi^{-1}$ will be the desired function of Theorem \ref{main}, with $P(f)=S(f)= \{t_n^*\}$ and $\psi(\tilde{h}(\tilde{s}_n)):=t_n^*$ now defined on all of $S$.%However we claim that in what follows, we may assume that $h$ is onto. Indeed, if $h$ is not onto, we may augment the sequence $S$ with auxiliary points to form a discrete sequence $\tilde{S}$, and extend $h$ to $\tilde{S}=\{\tilde{s}_n\}$ so that $h:\tilde{S}\rightarrow S$ is onto. Then we apply Lemma \ref{folding_final_} with $\{z_n\}=\{\tilde{s}_n\}$ and $\{w_n\}=\{t_n^*\}$ where $\{t_n^*\}$ is any sequence with $|t_n^* - h(\tilde{s}_n)| < \varepsilon$ for all $n$ and $t_n^*=h(\tilde{s}_n)$ if $h(\tilde{s}_n)=\pm1$. Then, again, if we can choose $\{w_n\}=\{t_n^*\}$  so that  $\phi^{-1}(t_n^*)=h(\tilde{s}_n)$, then $f:=g\circ\phi^{-1}$ will be the desired function of Theorem \ref{main}, with $P(f)=S(f)= \{t_n^*\}$ and $\psi(h(\tilde{s}_n)):=t_n^*$ now defined on all of $S$. }

\section{Existence of a Fixpoint}\label{fixpoint_section}

With the concluding remarks of Section 1 in mind, we will look for a fixpoint of a self-map of an infinite product of closed  Euclidean discs. This fixpoint will correspond to a quasiregular map $g$ so that $g\circ\phi^{-1}$ is the desired meromorphic function in the conclusion of Theorem \ref{main}. In the previous sections we built a graph $G$ and a quasiregular map $g$ associated to a pair $(S,h)$ from Theorem \ref{main}. The set of critical points of $g$ included $S$.

%The function $g$ depended on a choice of the images of each $s\in S$. We enumerate $S=(s_i)$, $S^*=(s_i^*)$ where we denote by $s_i^*$ the singular value of $g$ near $s_i$ (see Figure \ref{fig:example_of_S}). %It will be helpful here to assume, without loss of generality, that $h:S \rightarrow S$ is onto; the setting that $h$ is not onto can be reduced to this case by augmenting a given $S$ to some superset $\tilde{S}$ and extending $h:\tilde{S}\rightarrow \tilde{S}$ so that $h$ maps onto $S$. 

The function $g$ depended on a choice of the images of each $s\in S$. We enumerate $S=(s_i)$. For each choice of $(s_j^*)_{j=1}^\infty$ where $s_j^*\in \overline{D(\varepsilon, s_j)}$, Lemma \ref{folding_final_} (and the assumption that $h$ is onto - see Section \ref{Reducing}) gives some quasiregular function $g$ with critical points including $S$, and critical values at each $s_j^*$ where $g(s_i)=s_j^*$ if and only if $h(s_i)=s_j$.  Thus there is a corresponding quasiconformal map $\phi$ so that $g\circ\phi^{-1}$ is meromorphic. Moreover, we have noted that we can arrange for $|\phi(z)-z|<\varepsilon$ for all $z\in\mathbb{C}$, and $\phi(z)-z=o(1)$ as $z\rightarrow\infty$. Since $\phi(z)-z=o(1)$ as $z\rightarrow\infty$, we can now fix some positive sequence $\varepsilon_i\rightarrow0$ with $\varepsilon_i<\varepsilon$ over all $i$, such that $|\phi(s_i)-s_i|<\varepsilon_i$, and $(\varepsilon_i)$ is independent of a choice of $(s_i^*)$.

%, $S^*=(s_i^*)$ where we denote by $s_i^*$ the singular value of $g$ near $s_i$ (see Figure \ref{fig:example_of_S}). For each choice of $(s_i^*)_{i=1}^\infty$ where $s_i^*\in \overline{D(\varepsilon, s_i)}$, Lemma \ref{folding_final_} gives some quasiregular function $g$ with critical points including $S$, and critical values at $s_i^*$ where $g(s_i)=s_j^*$ if and only if $h(s_i)=s_j$.  Thus there is a corresponding quasiconformal map $\phi$ so that $g\circ\phi^{-1}$ is meromorphic. Moreover, we have noted that we can arrange for $|\phi(z)-z|<\varepsilon$ for all $z\in\mathbb{C}$, and $\phi(z)-z=o(1)$ as $z\rightarrow\infty$. Since $\phi(z)-z=o(1)$ as $z\rightarrow\infty$, we can now fix some positive sequence $\varepsilon_i\rightarrow0$ with $\varepsilon_i<\varepsilon$ over all $i$, such that $|\phi(s_i)-s_i|<\varepsilon_i$, and $(\varepsilon_i)$ is independent of a choice of $(s_i^*)$. 

\begin{lem} 
With notation as above, the  composition of maps
\begin{equation*}
	\prod_{i=1}^\infty \overline{D(\varepsilon_i,s_i)}
 \rightarrow L^\infty(\mathbb{C})
	\rightarrow \prod_{i=1}^\infty \overline{D(\varepsilon_i,s_i)}
 \end{equation*}
given by
\begin{equation}\label{fixpoint}
(s_i^*)\rightarrow\mu_{(s_i^*)}
\rightarrow(\phi_{\mu_{(s_i^*)}}(s_i))
 \end{equation}
is continuous between the product topologies.
\end{lem}

\begin{proof} 
We recall that a basis for the product topology on
$\prod_{i=1}^\infty \overline{D(\varepsilon_i,s_i)}$
is given 
by products $\prod_{i=1}^\infty U_i$ where each $U_i$
is open in $\overline{D(\varepsilon_i,s_i)}$ and 
$U_i=\overline{D(\varepsilon_i,s_i)}$ except for finitely many 
indices $i$. In particular the topology on 
$\prod_{i=1}^\infty \overline{D(\varepsilon_i,s_i)}$ is coarse,
 and so it is  easy 
to prove continuity of a map \emph{into}  
$\prod_{i=1}^\infty \overline{D(\varepsilon_i,s_i)}$; we only 
need to check continuity into each factor of the product
(see for example Theorem 19.6 of 
\cite{MR0464128}). This is precisely 
Theorem \ref{dependence}.
This gives the continuity of the second map in 
(\ref{fixpoint}).

 On the other hand, it is slightly more difficult to 
prove continuity of the first map in 
(\ref{fixpoint}).
%Here we will use  the normalization of  our
% quasiconformal maps $\phi$  so that
% $\phi(z)=z+O(1/|z|)$ near $\infty$.
 Fix some sequence $(s_i^*) \in \prod_{i=1}^\infty 
\overline{D(\varepsilon_i,s_i)}$ and an open neighborhood 
$D(r,\mu_{(s_i^*)})\subset L^\infty(\mathbb{C})$. 
We need to find some product of open sets 
$\prod_{i=1}^\infty U_i \ni (s_i^*)$ so that 
$U_i=\overline{D(\varepsilon_i,s_i)}$ except for finitely many 
indices $i$, and for any $(t_i)\in\prod_{i=1}^\infty U_i$
we have $||\mu_{(t_i)}-\mu_{(s_i^*)}
||_{L^\infty(\mathbb{C})}<r$. 

Suppose we have indexed $s_1^*$ as the unique asymptotic 
value of $g$, and consider some fixed $i>1$. Varying 
$s_i^*$ changes the dilatation of $g$ only on a 
collection of thin annuli $A_j\subset D_j$ for which
$h(s_j)=s_i$. Let $E_i$ be the union of those annuli
$A_j \subset D_j$ for which $h(s_j)=s_i$. For $s_1^*$,
we define $E_1$ as $T(r_0)\setminus(\cup_s D_s)$ in
union with any $A_j\subset D_j$ for which $h(s_j)=s_1$;
the reason for this definition is that varying $s_1^*$
changes the dilatation of $g$ only on $E_1$. 

Now consider some fixed $E_j$. Then $\mu_{(t_i)}|_{E_j}$ depends only on $t_j$, and as $t_j\rightarrow s_j^*$, it is clear that $\mu_{(t_i)}|_{E_j} \rightarrow \mu_{(s_i^*)}|_{E_j}$ uniformly, so that we may choose some $\delta>0$ with $||\mu_{(t_i)}-\mu_{(s_i^*)} ||_{L^\infty(E_j)}<r$ as long as $t_j\in \overline{D(\delta,s_j^*)}$. Since $\varepsilon_j\rightarrow0$, we know that for all sufficiently large $j$ we have $D(\varepsilon_j, s_j)\subset D(\delta,s_j^*)$, and we choose $U_j=\overline{D(\varepsilon_j,s_j)}$ for such $j$. For the finitely many other indices $j$, we choose $U_j=D(\delta, s_j^*)\cap \overline{D(\varepsilon_j,s_j)}$. With this choice of $\prod_{i=1}^\infty U_i$, we have that $||\mu_{(t_i)}-\mu_{(s_i^*)} ||_{L^\infty(\mathbb{C})}<r$ for any $(t_i)\in\prod_{i=1}^\infty U_i$, as required.
%Thus we have continuity of the map 
%\[ \prod_{i=1}^\infty D(\varepsilon_i,s_i)
% \rightarrow \prod_{i=1}^\infty D(\varepsilon_i,s_i)
%\]
%where 
%\[ 
%	(s_i^*)\rightarrow(\phi_{\mu_{(s_i^*)}}(s_i)).
%\]

\end{proof} 

We are now ready to prove Theorem \ref{main}. 
We remark that the singular values $\pm1$ are distinguished
 from other singular values of $g$ in that we do not 
have the freedom in Lemma \ref{folding_final_} to perturb
 $\pm1$. Therefore, the proof is simpler in the case
 when the correction map $\phi$ fixes both $\pm 1$. 
 We will first prove the theorem in this case, and 
 then extend the argument to cover the general case.

\begin{proof}[Theorem \ref{main} assuming 
	$\phi$ fixes $\pm 1$]
We have just shown that the map in (\ref{fixpoint})
is continuous and maps 
$\prod_{i=1}^\infty \overline{D(\varepsilon_i,s_i)}$ 
 into itself; 
this was arranged by definition of $\varepsilon_i$.
Thus Theorem \ref{fixed_point} implies this map 
has fixpoint.
A fixpoint of (\ref{fixpoint}) corresponds to some
choice of $(s_i^*)$ so that $\phi^{-1}_{\mu_{(s_i^*)}}
(s_i^*)=s_i$  for all singular values other than $\pm1$.
By assumption  this also holds if we set $(1)^* =1$
and $(-1)^*=-1$.
If $f(z) = g(\phi^{-1}(z))$, then $f$ is meromorphic 
and we have $f|_{S^*}=\psi\circ h\circ\psi^{-1}$, where
the map $\psi$ is defined by $\psi(s)=s^*$.
This proves Theorem 1 
 under the extra assumption that $\phi^{-1}$ fixes both $\pm1$.

\end{proof} 

\begin{proof}[Theorem \ref{main} in general]
Now we consider the case when the correction map
$\phi^{-1}$ does not necessarily  fix both $\pm1$.
Let $\delta >0$ be the distance from $\pm1$ to the 
remainder of the singular set (this is positive 
since the singular set is discrete). For each 
$x,y \in  \overline{D(-1, \delta/2)}
\times \overline{D(1,\delta/2)}$
let $\eta$ be a quasiconformal map so that 
$\eta(-1) = x$, $\eta(1)=y$ and $\eta$ is the 
identity outside 
$U=D(-1, \delta)  \cup D(1,\delta)$. Clearly we 
can do this with a dilation that is uniformly 
bounded independent of $\delta$, $x$ and $y$ and
is supported inside $U$.

We now wish to repeat the fixpoint argument 
above with $g$ replaced by $G= \eta \circ g$. This is 
still a quasi-regular map that depends on 
the parameters $\{ s_j^*\}$ and two new parameters
$x,y$. The $g$ preimages of  $U$ are contained 
in $T(r)$ for a uniform choice of $r$, so 
$G$ still has a dilatation that is  supported in a set
of small area and this area decays exponentially near 
$\infty$. Therefore the corresponding correction map 
$\phi$ still varies continuously in all the 
parameters. Moreover, $\phi$ will move each 
of the points $\pm 1$ by as little as we wish,
depending on our choice of $W$. Therefore we 
can arrange for $\phi$ to map $-1$ into  
$ \overline {D(-1,\delta/2)} $ 
for any $x \in   \overline {D(-1,\delta/2)} $, 
independent of how the other parameters are 
chosen. Similarly, $\phi$ maps $1$ into 
$\overline{  D(1, \delta/2)}$. 
The chosen closed disks around all 
the other (non $\pm1$)  singular values still map 
into themselves as before, so the fixpoint 
argument from above applies again.
More precisely, there is 
a choice of $x \in \overline {D(-1, \delta/2)}$,
$y \in \overline{ D(1, \delta/2)}$,
and 
$ \{s_j^*\} \in \overline{D(s_j, \varepsilon_j)} $
so that 
$\phi(-1)=x$, $\phi(1)=y$, and
 $ \phi(s_j) = s_j^*$ for all $j$.

Let $f(z) =  \eta(g(\phi^{-1}(z))) $. 
Then the definition of $\eta$ implies  
$$f(x) = \eta(g(\phi^{-1}(x))) = 
          \eta(g(-1)) = 
\begin{cases}
    x & \mbox{ if } g(-1)=-1, \\
    y & \mbox{ if } g(-1)=1, \\
   g(-1) & \mbox{ otherwise. }
\end{cases} 
$$
Similarly for $f(y)$. For other $s_j \in S$ such that   $s_j \not = \pm 1$, 
we have $\eta(s_j) =s_j$ (since $\eta$ is the 
identity away from $\pm 1$), so 
$$f(s_j^*) = \eta(g(\phi^{-1}(s_j^*))) = 
          \eta(g(s_j)) = 
\begin{cases}
    x & \mbox{ if } g(s_j)=-1, \\
    y & \mbox{ if } g(s_j)=1, \\
   g(s_j) & \mbox{ otherwise. }
\end{cases} 
$$
Thus if we set $\psi(-1)=x$, $\psi(1)=y$ and 
$\psi(s_j) = s_j^*$  for the other singular values, 
then  $f, \psi $ satisfy the 
conclusions of Theorem \ref{main}.

\end{proof}

%\begin{acknowledgements} 
%The authors would like to thank the anonymous referees for their suggestions which have led to an improved version of the paper.
%\end{acknowledgements}

%{\color{red} \subsection*{Acknowledgements}
%\addtocontents{toc}{\protect\setcounter{tocdepth}{1}}

%The authors would like to thank the anonymous referees for their suggestions which have led to an improved version of the paper. 
% }

%\bibliography{bibfile2}   % name your BibTeX data base

%\begin{thebibliography}{}
% etc
%\end{thebibliography}

\end{document}